\documentclass[a4paper,11pt]{amsart}
\usepackage[utf8]{inputenc}

\usepackage{listings}
\usepackage{xcolor}

\usepackage{comment}
 \usepackage{hyperref} 
\usepackage{lipsum}
\usepackage{amsthm}
\usepackage{amsmath}
\usepackage{amsbsy}
\usepackage{bigints}
\usepackage{amsfonts}%
\usepackage{amssymb}%
\usepackage{graphicx}
\usepackage[utf8]{inputenc}%
\usepackage[T1]{fontenc}%
\usepackage[english]{babel}%
\usepackage[all]{xy}
\usepackage{amsmath, amsthm, amssymb,amsfonts,mathrsfs,amscd,latexsym}
\usepackage{graphicx}
\usepackage{tikz-cd}
\usepackage{mathtools}
\usepackage{hyperref}
\usepackage{graphicx}
\graphicspath{ {./image/} }

\newcommand{\poubelle}[1]{}

\usepackage{wrapfig}

\usepackage[all]{xy}
\usepackage{colortbl}

\usepackage{amsmath,amssymb,graphics,epsfig,color}
\usepackage{enumerate,wasysym,stmaryrd}
\usepackage[new]{old-arrows}
\usepackage{color}
\usepackage{fancybox}
\usepackage{colordvi}
\usepackage{multicol}
\usepackage{colordvi}
\usepackage{geometry}
\usepackage{lscape}
\usepackage{amsthm}
\usepackage{eqnarray,amsmath}
\usepackage{array}
\usepackage{booktabs}   
  \newcolumntype{x}[1]{>{\centering\hspace{0pt}}p{#1}}
\def\R{\mathbb{R}}
\def\K{\mathbb{K}}

\def\{{\left\lbrace}
\def\}{\right\rbrace}

\newtheorem{theorem}{Theorem}[section]

\newtheorem{corollary}[theorem]{Corollary}
\newtheorem{define}[theorem]{Definition}

\newtheorem{lemma}[theorem]{Lemma}
\newtheorem{prop}[theorem]{Proposition}
\theoremstyle{definition}
\newtheorem{remark}[theorem]{Remark}

\renewcommand{\v}{\mathrm{v}}
\newcommand{\w}{\mathrm{w}}

\newcommand{\red}{\mathrm{red}}

\newcommand{\PL}{\textnormal{P,\textbf{L}}}
\newcommand{\ext}{\mathrm{ext}}

\newcommand{\T}{\mathbb{T}}

\renewcommand{\P}{\mathbb{P}}
\newcommand{\Z}{\mathbb{Z}}
\renewcommand{\K}{\mathbb{K}}

\newcommand{\DF}{\mathcal{F}}
\newcommand{\test}{\mathtt{test}}

\definecolor{codegreen}{rgb}{0,0.6,0}
\definecolor{codegray}{rgb}{0.5,0.5,0.5}
\definecolor{codepurple}{rgb}{0.58,0,0.82}
\definecolor{backcolour}{rgb}{0.95,0.95,0.92}

\lstdefinestyle{mystyle}{
  backgroundcolor=\color{backcolour}, commentstyle=\color{codegreen},
  keywordstyle=\color{magenta},
  numberstyle=\tiny\color{codegray},
  stringstyle=\color{codepurple},
  basicstyle=\ttfamily\footnotesize,
  breakatwhitespace=false,         
  breaklines=true,                 
  captionpos=b,                    
  keepspaces=true,                 
  numbers=left,                    
  numbersep=5pt,                  
  showspaces=false,                
  showstringspaces=false,
  showtabs=false,                  
  tabsize=2
}

\title[Effective weighted stability]{An effective weighted K-stability condition for polytopes and semisimple principal toric fibrations}

\author[T. Delcroix]{Thibaut Delcroix} 
\address{Thibaut Delcroix, Univ Montpellier, CNRS, Montpellier, France}
\email{thibaut.delcroix@umontpellier.fr}

\author[S. Jubert]{Simon Jubert} 
\address{Simon Jubert \\ Departement de Mathématiques,  UQAM, C.P. 8888, Succursale Centre-ville, Montréal
(Quebec), H3C 3P8, Canada}

\address{  Institut de Math{\'e}matiques de Toulouse\\   Université Paul Sabatier \\ 118 route de Narbonne\\  31062 Toulouse \\  France \\ } 
\email{simonjubert@gmail.com}



\begin{document}

\maketitle

\begin{abstract}
The second author has shown that existence of extremal Kähler metrics on semisimple principal toric fibrations is equivalent to a notion of weighted uniform K-stability, read off from the moment polytope. 
The purpose of this article is to prove various sufficient conditions of weighted uniform K-stability which can be checked effectively and explore the low dimensional new examples of extremal Kähler metrics it provides. 
\end{abstract}

\section{Introduction}

Calabi's work has been extremely influential in Kähler geometry, his name being still associated to some of the most fundamental objects of interest. 
The present article is motivated by two of these, Calabi's extremal Kähler metrics and Calabi's ansatz. 

Extremal Kähler metrics provide a natural notion of canonical Kähler metrics in a given Kähler class on a compact Kähler manifold \(X\): they are the metrics that achieve the minimum of the \(L^2\)-norm of the scalar curvature. 
Kähler metrics with constant scalar curvature (cscK metrics for short) are special cases of such metrics, but Calabi showed in \cite{Cal82} that there may exist extremal Kähler metrics when there exists no cscK metrics at all, by exhibiting extremal Kähler metrics on Hirzebruch surfaces. 
In order to show this, Calabi relied on the simple yet powerful idea that one should search for extremal Kähler metrics among those Kähler metrics that behave well with respect to the geometry of the manifold. 

This was not a new idea of course. Matsushima showed for example \cite{Mat57} that cscK metrics must behave well with respect to biholomorphism. 
More precisely, the automorphism group of \(X\) must be the complexification of the isometry group of the cscK metric, if it exists. 
This is preventing Hirzebruch surfaces from admitting cscK metrics as their automorphism group is non-reductive. 
 
Calabi went further and restricted to metrics that respect the structure of \(\mathbb{P}^1\)-bundles of Hirzebruch surfaces. 
He was then able to translate, for such metrics, the extremal property into a simple ODE and to solve it, showing the existence of extremal Kähler metrics. 
His construction was later referred to as Calabi's ansatz, used in various situations and generalized in various directions. 
It would be easy to fill pages with a bibliographical review of these, but it is not the purpose of this introduction. 
We only stress that a common theme is usually the desire to get explicit existence results or criterions. 
An influential illustration is \cite{ACGTIII}, where a variant of Calabi's ansatz was used to show that on various \(\mathbb{P}^1\)-bundles, existence of extremal Kähler metrics reduces to checking the positivity of a polynomial on \([-1,1]\), the so-called extremal polynomial. 
In the series of papers leading to \cite{ACGTIII} (see also \cite{HS}), the general idea of Calabi's ansatz was actually pushed way further, allowing for example to consider certain fibrations with toric fiber. 

The interest for such fibrations was significantly renewed last year, when the second author proved in \cite{SJ}, using the breakthrough results of Chen and Cheng \cite{CCIII}, that a uniform version of the Yau--Tian--Donaldson conjecture holds for semisimple principal toric fibrations, a very large class of toric fibrations. 
While it allows to translate the question of existence of extremal Kähler metrics on such manifolds into a question of convex geometry on their moment polytopes, it is not yet an explicitly checkable criterion, as the conditions to check still form an infinite dimensional space. 
Motivated by the more practical philosophy behind Calabi's ansatz, we prove in the present paper various sufficient conditions of existence of extremal Kähler metrics which may be easily checked. 
Our approach is based on an initial idea by Zhou and Zhu \cite{ZZ}, exploited in greater generality by the first author in \cite{KSSV2}. 

The previous paragraphs are meant as an introduction to our results, and it should be stressed that it presents as such a biased and very incomplete historical account of the study of extremal Kähler metrics on manifolds with large symmetry. We refer to Székelyhidi's book \cite{SzeBook}, Gauduchon's lecture notes \cite{Gau} for a general introduction to extremal Kähler metrics, and to Donaldson's remarkable survey \cite{Don08} and Apostolov's lecture notes \cite{ApoLN} for some of the more specific aspects of manifolds with large symmetry. 
More recent developments very related to our work will be discussed at the beginning of Section~\ref{section-geometric}.  

Let us now highlight in the remainder of this introduction our main results. 
For this, a few notations are needed, and the full details will be given in Section~\ref{section-geometric}. 
Semisimple principal toric fibrations are certain holomorphic fiber bundles \(\pi: Y\to B\) where the basis \(B = \prod_a B_a\) is a product of Hodge manifolds \((B_a,\omega_a)\) with constant scalar curvature \(s_a\), and where the fiber \(X\) is toric under a compact torus \(\mathbb{T}\). 
They are constructed from certain types of principal \(\mathbb{T}\)-bundles, essentially determined by the data of a tuple \((p_a)\) of one-parameter subgroups of \(\mathbb{T}\). In this paper, a one-parameter subgroup \(p_a: \mathbb{S}^1\to \mathbb{T}\) of \(\mathbb{T}\) will be identified with the element of the Lie algebra of \(\mathbb{T}\) determined by the image of \(1\in \mathbb{R}=\text{Lie}(\mathbb{S}^1)\) under the differential of \(p_a\) at the neutral element. 
In particular, it defines a linear function on the dual of the Lie algebra of \(\mathbb{T}\). 
On such manifolds, a Kähler class is called compatible if it decomposes as the sum of a relative Kähler class induced by a Kähler class \([\omega_X]\) on \(X\), and a sum of real multiples \(c_a\pi^*[\omega_a]\) of the pull-backs of the Kähler classes \([\omega_a]\). 
An admissible Kähler class contains admissible Kähler metrics, that behave well with respect to the fibration structure. 

\begin{theorem}
\label{thm-intro}
Assume that \(Y\) is a semisimple principal toric fibration, that the toric fiber \(X\) is Fano equipped with the Kähler class \([\omega_X]=t 2\pi c_1(X)\), and let \([\omega_Y]\) be an admissible Kähler class. 
Assume that for all \(a\), \(2\dim(B_a)c_a\geq ts_a\) and that at every vertex \(x\) of the moment polytope \(P\) of \((X,[\omega_X])\), 
\[ 2(\dim(Y)+1)+\sum_a \frac{ts_a-2\dim(B_a)c_a}{p_a(x)+c_a} -t l_{\ext}(x) \geq 0 \]
where \(l_{\ext}\) is the extremal affine function. 
Then there exists an extremal Kähler metric in \([\omega_Y]\).  
\end{theorem}

Recall that, when a maximal torus of automorphisms of \(Y\) is fixed, the scalar curvature of an invariant extremal Kähler metric, if it exists, is a holomorphy potential of a well defined vector field called the extremal vector field. 
In the statement above, the extremal function is encoding the extremal vector field, and a choice of maximal torus of automorphisms of \(Y\) is implicitly assumed. We will explain why it reduces to an affine function on the polytope \(P\) in Section~\ref{section-geometric}. 

We actually prove a much more general sufficient condition, Theorem~\ref{prop-combinatorial}, that does not require the fiber to be Fano. 
Since we obtain already a wealth of new examples with this particular case, and it is a natural generalization of the \(\mathbb{P}^1\)-bundle case, we focus on this result for the introduction. 

In the case when the fibration is Fano itself, and not only its fiber, then \(t=1\) and \(s_a=2\dim(B_a)c_a\) so one gets a particularly simple criterion:
\begin{corollary}
\label{cor-intro}
A Fano semisimple principal toric fibration \(Y\) admits an extremal Kähler metric in \(c_1(Y)\) if its extremal affine function \(l_{\ext}\) satisfies 
\[ \sup l_{\ext} \leq 2(\dim(Y)+1) \]
and the latter obviously needs only be verified at vertices of the moment polytope. 
\end{corollary}

We provide, for the reader's convenience, an elementary Python program implementing the sufficient condition from Theorem~\ref{thm-intro} in the case when there is only one factor in \(B\) and the fiber is of dimension one or two. 
It would be easy to imitate these to allow greater flexibility in the data. 
It may be used either with all the data given numerically, or some of the data treated as variable. 
We use this to our advantage to prove the existence of extremal Kähler metrics in a wide range of Kähler classes for some examples of fibrations. 

\begin{prop}
\label{prop-intro}
Let $Y=\P_B(\mathcal{O}_B \oplus H^{-p_1} \oplus H^{-p_2})$, where \(B\) is a Kähler-Einstein Fano threefold, \(H\) is the smallest integral divisor of \(2\pi c_1(B)\) and $1 \leq p_1 \leq p_2 $. Then there exists an extremal Kähler metric in the Kähler class \( c_1(X)+\lambda  c_1(B)\) for $\lambda \geq 7 p_2$, where \(c_1(X)\) and \(c_1(B)\) respectively denote the relative first Chern class and the pull-back of the first Chern class, by an abuse of notations. 
\end{prop}

Here, \(Y\) is a semisimple principal fibration over the base \(B\), with fiber the projective space \(X=\mathbb{P}^2\). 
More generally, projectivizations of direct sums of line bundles can often be considered as semisimple principal fibrations, as explained in Section~\ref{section-projective-bundles}. 

The article is organized as follows. 
In Section~\ref{section-labelled-polytopes}, we prove a general sufficient condition for weighted uniform K-stability of labelled polytopes, and consider the special case of monotone polytopes. 
Section~\ref{section-geometric} explains the geometric origin of weighted uniform K-stability of labelled polytopes, with a particular emphasis on semisimple principal toric fibrations. 
In Section~\ref{section-applications}, we put together the two aspects to prove Theorem~\ref{thm-intro} and Corollary~\ref{cor-intro} using the monotone case of Section~\ref{section-labelled-polytopes}, as well as more general statements. 
We present various examples of applications of the sufficient condition in Section~\ref{section-examples}, including Proposition~\ref{prop-intro}. 
Finally, we include in an appendix elementary Python programs computing the sufficient condition for fibrations with only one factor in the basis, and a one or two dimensional Fano fiber. 

\subsection*{Acknowledgements}
The authors are very grateful to Vestislav Apostolov and Eveline Legendre for their valuable comments on the manuscript. 
The authors also thank the referee for his useful suggestions and corrections.   
 The first author is partially funded by ANR-21-CE40-0011 JCJC project MARGE and ANR-18-CE40-0003-01 JCJC project FIBALGA, as well as PEPS JCJC INSMI CNRS projects 2021 and 2022. The second author was supported by PhD fellowships of the UQAM and of the Université de Toulouse III - Paul Sabatier.

\section{Weighted K-stability of labelled polytopes: a sufficient condition}
\label{section-labelled-polytopes}

\subsection{Weighted K-stability of labelled polytopes}

Let \(V\) be an affine space of dimension \(\ell\), equipped with a fixed Lebesgue measure \(dx\). 
A \emph{labelled polytope} in \(V\) is a pair \((P,\textbf{L})\) where \(P\) is a (compact, convex) polytope in \(V\) and \(\textbf{L}=(L_j)_{j=1}^d\) is a minimal set of defining affine functions for \(P\), that is, 
\[ P = \{ x\in V \mid \forall j, \quad L_j(x) \geq 0\} \]
where \(d\) is the number of facets (codimension one faces) of \(P\). 
We denote by \(F_j:= \{ x\in P \mid L_j(x) = 0\} \) the facet of \(P\) defined by \(L_j\).

\begin{define}
The \emph{labelled boundary measure} \(d\sigma\) is the measure on \(\partial P\) whose restriction to the facet \(F_j\) is defined by 
\( dL_j \wedge d \sigma = -dx. \)
\end{define}

Note that the labelled boundary measure depends heavily on the choice of labelling \((L_j)\). 
For example, for any tuple \((r_j)\) of positive real numbers, the tuple \((L_j')=(r_jL_j)\) is another labelling of \(P\). 
The associated labelled boundary measure \(d\sigma'\) satisfies \(d\sigma' = \frac{1}{r_j}d\sigma\) on \(F_j\). 
In particular, if the \(r_j\) are not all equal, there is no obvious relation between \(d\sigma\) and \(d\sigma'\). 
Similarly, the notion of weighted uniform K-stability that we are about to define depends heavily on the choice of labelling. 

Let \(\v\) be a continuous, positive function on \(P\), and let \(\w\) be a continuous function on \(P\). 
Following \cite{SKD, AL, LLS}, we define the \emph{$(\v,\w)$-Donaldson--Futaki invariant} of the labelled polytope $(P,\textbf{L})$ as the functional 
\(\DF\) on the space of continuous functions on \(P\) such that
\begin{equation}{\label{define-futaki}}
  \DF(f):=  2\int_{\partial P} f(x)\v(x) d\sigma -  \int_P  f(x)  \w(x)  dx.
\end{equation}
Following \cite{NS,TH}, we also set 
\[ 
\lVert f \rVert_J = \inf_{l\in \mathrm{Aff}(V)} \int_P \left(f+l-\inf_P(f+l)\right) dx 
\]
where \(\mathrm{Aff}(V)\) denotes the space of affine functions on \(V\). 

\begin{define}
A labelled polytope \((P,\textbf{L})\) is \((\v,\w)\)-uniformly K-stable if there exists a \(\lambda >0\) such that for any continuous convex functions \(f\) on \(P\), 
\begin{equation}
\label{general-stability-polytope}
\DF(f) \geq \lambda \lVert f \rVert_J.
\end{equation} 
\end{define}

\begin{remark}
Note that \(\DF\) is linear, and the right-hand side of \eqref{general-stability-polytope} is always non-negative, hence the following is a necessary condition for \eqref{general-stability-polytope} to hold: 
\begin{equation}{\label{annulation-Futaki}}
    \forall f\in \mathrm{Aff}(V), \DF(f)=0.
\end{equation}
\end{remark}

We will explain in Section~\ref{section-geometric} the geometric significance of this notion for various choices of \(\v\) and \(\w\), let us for now just highlight that when \(\v\) and \(\w\) are constant, the functional \(\DF\) first appeared in \cite{SKD} as an expression of the (Donaldson--)Futaki invariant for toric test configurations in the study of K-stability of toric manifolds, whence the name. 

We denote by \(\mathcal{CV}^0(P)\) the space of continuous convex functions on \(P\), and by \(\mathcal{CV}^1(P)\) the space of all convex functions \(f\) on \(P\) which are the restrictions to \(P\) of a continuously differentiable function defined on an open subset of \(V\) containing \(P\). 
Note that by uniform approximation by smooth functions, it is enough to consider only functions in \(\mathcal{CV}^1(P)\) to check condition \eqref{general-stability-polytope}. 

 In order to deal more efficiently with the right hand side of~\eqref{general-stability-polytope}, following \cite{SKD}, we consider the following normalization of functions.
 We choose a point \(x_0\) in the interior \(P^0\) of the polytope \(P\). 
It allows to choose a linear complement \(\mathcal{CV}_*^1(P)\) to \(\mathrm{Aff}(V)\) in \(\mathcal{CV}^1(P)\), defined by 
\begin{equation}{\label{normalized-function-polytope}}
     \mathcal{CV}^{1}_*(P):=\{ f \in \mathcal{CV}^{1}(P) \mid \forall x, f(x) \geq f(x_0)=0 \}.
\end{equation}
Then, any $f\in \mathcal{CV}^{1}(P)$ can
be written uniquely as $f = f^*+f_0$, where $f_0$ is affine and \(f^*\in \mathcal{CV}_*^1(P)\), and we will use these notations in the following. 
By linearity, \(\DF(f) = \DF(f^*)\) if \(\DF\) vanishes on \(\mathrm{Aff}(V)\). 

\begin{lemma}
{\label{uniform-K-stable}}
The labelled polytope \((P,\textbf{L})\) is \((\v,\w)\)-uniformly K-stable if and only if 
there exists $\lambda > 0$ such that for all  $f \in \mathcal{CV}^{1}(P)$, 
\begin{eqnarray}{\label{uniform-equation}}
\DF(f) \geq \lambda \lVert f^*\rVert_{L_1}
\end{eqnarray}
where $\lVert \cdot \rVert_{L_1}$ denotes the $L^1$-norm on $P$ with respect to the Lebesgue measure \(dx\). 
\end{lemma}

\begin{proof}
From \cite[Proposition~4.1~(3)]{NS}, there exists a constant \(C_1>0\) such that for all continuous convex functions on \(P\), 
\[ \lVert f \rVert_J \leq \lVert f \rVert_{L^1} \leq C_1\lVert f \rVert_J \]
The equivalence between condition~\eqref{general-stability-polytope} and condition~\eqref{uniform-equation} follows immediately. 
\end{proof}

\begin{remark}
Condition~\eqref{uniform-equation} is the condition that we will effectively use in the sequel, so one might wonder why we introduced the first definition. 
The point is that by Lemma~\ref{uniform-K-stable}, condition~\eqref{uniform-equation} is independent of the choice of \(x_0\), and condition~\eqref{general-stability-polytope} makes it perfectly clear.
In the more familiar unweighted case, the equivalence between various notions of K-stability of polytopes was fully worked out by Nitta and Saito \cite{NS}. 
\end{remark}

\subsection{The sufficient condition}

Let  \(\mathcal{C}^{0}(P,\R)\) denote the space of continuous functions on \(P\), and let \(\mathcal{C}^{1}(P,\R)\) denote the space of functions that are the restriction to \(P\) of continuously differentiable functions defined in an open subset of \(V\) containing \(P\). 

Recall that \(F_j\) denotes the facet of \(P\) defined by \(L_j\). 
For each \(j\), let \(P_j\) be the cone with basis \(F_j\) and vertex \(x_0\) as illustrated in Figure~\ref{figure-cone}.
For a function \(f\in \mathcal{C}^1(P,\R)\), we denote by \(d_xf\) its differential at \(x\in P\). 
The following is the main technical result of our paper, it imitates quite closely part of the proof by Zhou and Zhu \cite{ZZ} of a coercivity criterion for the modified Mabuchi functional on toric manifolds. 

\begin{figure}
\label{figure-cone}
\centering
\caption{The cone decomposition}
\begin{tikzpicture}  
\draw (0,0) node{$\bullet$};
\draw (0,0) node[left]{\(x_0\)};
\draw [thick] (-2,-2) -- (2,-2) -- (2,0) -- (0,2) -- (-2,2) -- cycle;
\draw (-1,2) node[above]{$F_j$};
\draw (-0.5,1.2) node{$P_j$};
\draw [dashed] (0,0) -- (-2,-2);
\draw [dashed] (0,0) -- (2,-2);
\draw [dashed] (0,0) -- (2,0);
\draw [dashed] (0,0) -- (0,2);
\draw [dashed] (0,0) -- (-2,2);
\end{tikzpicture}
\end{figure}
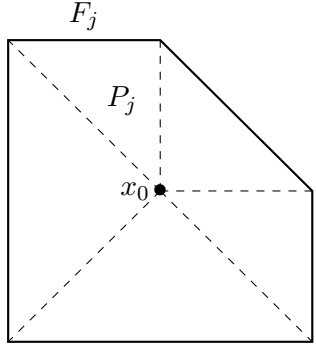

\begin{theorem}{\label{prop-combinatorial}}
Let $\v \in \mathcal{C}^1(P,\R)$ be a positive function on \(P\) and let \(\w \in \mathcal{C}^0(P,\R)\). Assume that \(\DF\) vanishes on \(\mathrm{Aff}(V)\) and that for all $j=1, \dots, d$, for all \(x \in P_j\), 
\begin{equation}{\label{combinatorial1}}
    \frac{1}{L_j(x_0)}\left(\v(x)( \ell+1) + d_x \v(x-x_0) \right) - \frac{\w(x)}{2} \geq 0,
\end{equation}
then $(P,\textbf{L})$ is $(\v,\w)$-uniformly K-stable. 
\end{theorem}

\begin{proof}
Since \(L_j(x)=0\) for \(x\in F_j\), we have \(L_j(x_0)=d_xL_j(x_0-x)\) for all \(x\in F_j\). In particular, 
\[ \int_{F_j} f(x) \v(x) d\sigma = \int_{F_j} f(x)\v(x) \frac{-d_xL_j(x-x_0)}{L_j(x_0)}  d \sigma. \]
For each facet \(F\) of \(\partial P_j\) different from \(F_j\), and \(x\in F\), the interior product \(\iota_{x-x_0} (dx)\) vanishes since it vanishes on the affine space spanned by \(F\).  
If we further use that \(-dL_j\wedge d\sigma = dx\) on \(F_j\), we obtain  
\[ \int_{F_j} f(x) \v(x) d\sigma = \frac{1}{L_j(x_0)}\int_{\partial P_j} f(x)\v(x) \iota_{x-x_0} \left(dx\right). \]
Hence by Stokes theorem we obtain
\[ \int_{F_j} f(x) \v(x) d\sigma = \frac{1}{L_j(x_0)} \int_{P_j} \left( \v(x) d_xf(x-x_0) + \ell f(x) \v(x) + f(x) d_x\v(x-x_0) \right) dx. 
\]
Summing the previous identities over $j$ we get
\begin{equation}{\label{alternativ-expression}}
\begin{split}
\DF(f)=&\sum_{j=1}^d\frac{2}{L_j(x_0)}  \int_{P_j} \left( d_x f(x-x_0) - f(x) \right)\v(x)dx \\
+&\sum_{j=1}^d \int_{P_j} \left( \frac{2}{L_j(x_0)}\left((\ell+1) \v(x) + d_x\v(x-x_0)  \right) - \w(x)  \right) f(x) dx.
\end{split}
\end{equation}

 Assume condition~\eqref{combinatorial1} is satisfied and $(P,\textbf{L})$ is not $(\v,\w)$-uniformly K-stable. 
 
 We will show contradiction to a stronger condition than condition~\eqref{uniform-equation}. 
 Namely, assume that condition~\eqref{combinatorial1} is satisfied and that there exists a sequence of $\{f_k\}_{k \in \mathbb{N}}$ in $\mathcal{CV}^{1}(P)$ such that
\begin{equation}{\label{prop-sequence}}
    \lim\limits_{k \rightarrow \infty} \DF(f^*_k)=0 \qquad \text{and} \qquad \forall k\in \mathbb{N}, \int_{\partial P} \v(x) f^*_k(x) \mathop{d\sigma}=1.
\end{equation}
 
 Recall from \cite[Lemma~5.1.3]{SKD} (see \cite[Proposition~5.1.2]{NS} for a detailed proof and explicit constant \(C\)) that there exists a positive constant \(C>0\) such that for all \(f\in \mathcal{CV}^1(P)\), 
 \[ \int_{\partial P} f^*(x) \mathop{d\sigma} \geq C \lVert f^* \rVert_{L^1}. \]
 As a consequence, since \(\v>0\) on \(P\), there exists a constant \(C'>0\) such that for all \(f\in \mathcal{CV}^1(P)\), 
 \[ \int_{\partial P} \v(x) f^*(x) \mathop{d\sigma} \geq C' \lVert f^* \rVert_{L^1}. \]
 In particular, the sequence \(\{f^*_k\}\) has bounded \(L^1\) norm. 
 By \cite[Corollary 5.2.5]{SKD},  
 $\{f^*_k\}_{k\in \mathbb{N}}$ converges (up to a sub-sequence, still denoted by $f^*_k$) locally uniformly in $P^0$ to a convex function $f^*_{\infty}$ which still satisfies \(\inf f^*_{\infty} = f^*_{\infty}(x_0) = 0\). Since in addition all $f^*_k$ are smooth and convex we have $d_x f^*_k(x-x_0) - f^*_k(x)  \geq  0$. Then, since condition~\eqref{combinatorial1} is assumed to hold, all terms of the sum in \eqref{alternativ-expression} are non-negative. 
Evaluating~\eqref{alternativ-expression} at $f^*_k$ and passing to the limit reveals that  $\lim_{k\rightarrow \infty} d_x f^*_k(x-x_0) - f^*_k(x)  = 0$ almost everywhere in $P^0$, showing that $f^*_{\infty}$ is affine on $P^0$. Using again  \(\inf f^*_{\infty} = f^*_{\infty}(x_0) = 0\), we conclude that $f^*_{\infty}$ is the zero function on $P^0$.  

The local uniform convergence of \(\{f^*_k\}\) to the zero function shows that 
\[ \lim_{k\to \infty} \int_P f^*_k (x)\w(x) \mathop{dx} = 0 \]
that is, 
\[ \lim_{k\to \infty} 2\int_{\partial P} f^*_k(x) \v(x) \mathop{d\sigma} - \DF(f^*_k) = 0 \]
which is in contradiction with condition~\eqref{prop-sequence}. 

From this contradiction it follows that there exists a constant \(\mu>0 \) such that for all \(f\in \mathcal{CV}^1(P)\), 
\begin{align*} 
\DF(f) & \geq \mu \int_{\partial P} \v (x)f^*(x) \mathop{d\sigma} \\
& \geq \mu C' \lVert f^* \rVert_{L^1} 
\end{align*}
which concludes the proof. 
\end{proof}

\begin{remark}
\label{remark-varying}
We stress that the property of \((\v,\w)\)-uniform K-stability is independent of the choice of \(x_0\in P^0\) in the previous section, but condition~\eqref{combinatorial1} \emph{depends} on that choice. 
It is possible and useful in practical uses of the condition to vary this \(x_0\) according to the data of the problem, see \ref{section-varying}.
\end{remark}

\begin{remark}
\label{remark-continuous}
Condition~\eqref{combinatorial1} depends continuously on the labelled polytope, the weights \(\v\) and \(\w\), and the choice of \(x_0\).
\end{remark}

\subsection{Case of monotone polytopes}

Let us recall the terminology of monotone polytopes, used in \cite{Leg}. 

\begin{define}
A labelled polytope \((P,\textbf{L})\) is \emph{monotone} if there exists an \(x_0 \in P^0\) such that \(L_1(x_0)=L_2(x_0)=\cdots=L_d(x_0)\). 
\end{define}


There is thus an obvious choice of \(x_0\) in that case. Our sufficient condition indeed becomes much simpler in that case, since the decomposition of the polytope may essentially be forgotten. 

\begin{corollary}
\label{corollary-monotone}
Let \((P,\textbf{L})\) be a monotone labelled polytope with \(L_1(x_0)=L_2(x_0)=\cdots=L_d(x_0)=t\). Let $\v \in \mathcal{C}^1(P,\R)$ such that \(\v\) is positive on \(P\) and let \(\w \in \mathcal{C}^0(P,\R)\). Assume that \(\DF\) vanishes on \(\mathrm{Aff}(V)\) and that for all \(x \in P\), 
\begin{equation}\label{condition-monotone}
    \frac{1}{t}\left(\v(x)( \ell+1) + d_x \v(x-x_0) \right) - \frac{\w(x)}{2} \geq 0,
\end{equation}
then $(P,\textbf{L})$ is $(\v,\w)$-uniformly K-stable. 
\end{corollary}

The conditions involved form a finite set of conditions to check, contrary to the definition of \((\v,\w)\)-uniform K-stability. 
It is furthermore easy to implement in a computer program, \emph{via} formal or numerical computations depending on the data \((P,\textbf{L},\v,\w)\). 
The same is true for the more general Theorem~\ref{prop-combinatorial}, but the decomposition in cones makes it a bit more tedious.

\section{Geometric origin of weighted K-stability of polytopes}
\label{section-geometric}

\subsection{Weighted cscK toric manifolds}

The results from Section~\ref{section-labelled-polytopes} are motivated by the study of the existence of weighted cscK metrics on toric manifolds, as studied in \cite{SJ}. 

Let $\mathbb{T}$ be an $\ell$-dimensional compact torus. We denote by $\mathfrak{t}$ its Lie algebra and by $\Lambda \subset \mathfrak{t}$ the lattice of generators of circle subgroups, so that $\mathbb{T} = \mathfrak{t}/2\pi \Lambda$. 
Let $(X,\omega,\mathbb{T})$ be a compact Kähler toric manifold. 
Denote by \(\mu\) the moment map of \(X\) with respect to the action of \(\mathbb{T}\), and let $P=\mu(X) \subset \mathfrak{t}^*$ be the moment polytope. 
The polytope \(P\) is a Delzant polytope \cite{TD}, and in particular, there is a natural choice of labelling \(\mathbf{L}\) of \(P\) such that all the differentials \(dL_j\) of the defining affine functions \(L_j\) are primitive elements in the lattice \(\Lambda\). 

\begin{remark}
We focus here on smooth manifolds, but let us mention that the cases of orbifolds or pairs would also be natural settings to consider. In these situations, the labelling could be more general, thus justifying our choice to allow arbitrary labellings in the previous section. 
\end{remark}

In the context of toric manifold, the $\v$-weighted scalar curvature was introduced in \cite{LLS}. To avoid introducing too much notation, we give the definition of \cite{AL}, which makes sense for general compact K\"ahler manifold and coincide with the one of \cite{LLS} in the toric context. 
Let \(\mathcal{C}^{\infty}(P,\mathbb{R})\) denote the space of restrictions to  \(P\) of smooth functions defined on an open set containing \(P\), and \(\mathcal{C}^{\infty}(P,\mathbb{R}_{>0})\) the subspace of positive functions.

\begin{define}[Weighted cscK metrics]{\label{weighted-cscK}}

\noindent
\begin{enumerate}
    \item For \(\v \in \mathcal{C}^{\infty}(P,\mathbb{R}_{>0})\), define the \(\v\)-scalar curvature of \(\omega\) as the function 
    
\begin{equation*}
    Scal_{\mathrm{v}}(\omega):=\mathrm{v}(\mu)Scal(\omega)+ 2 \Delta_{\omega}\big(\mathrm{v}(\mu)\big) +  \textnormal{Tr}\big(G_{\omega} \circ (\textnormal{Hess}(\mathrm{v}) \circ \mu )\big),
\end{equation*}

\noindent  where  $Scal(\omega)$ is the usual scalar curvature of the Riemannian metric $g_{\omega}$ associated to $\omega$, $\Delta_{\omega}$ is the Riemannian Laplacian of $g_{\omega}$, $\textnormal{Hess}(\mathrm{v})$ is the Hessian of $\v$ viewed as a bilinear form on $\mathfrak{t}^*$  whereas  $G_{\omega}$ is the bilinear form with smooth coefficients on $\mathfrak{t}$, given by the restriction of $g_{\omega}$ on fundamental vector fields.
 
    \item If furthermore \(\w\in \mathcal{C}^{\infty}(P,\mathbb{R})\), then \(\omega\) is a \((\v,\w)\)-cscK metric if 
    \[ \operatorname{Scal}_{\v}(\omega)=\w\circ\mu \]
\end{enumerate}
\end{define}

In general, no YTD correspondence is proved for the existence of weighted cscK metrics on toric manifolds. 
However by analogy with the unweighted cscK case, there is a known candidate for the corresponding K-stability condition, which translates on the polytope as Definition~\ref{general-stability-polytope}. 
In fact, the direction from existence of weighted cscK metrics to K-stability was proved in general by Li, Lian, Sheng \cite{LLS17}. 

\begin{theorem}[{\cite[Theorem 2.1]{LLS17}}]
\label{cscK-to-stability}
If \(\omega\) is a \((\v,\w)\)-cscK metric, then \((P,\mathbf{L})\) is \((\v,\w)\)-uniformly K-stable. 
\end{theorem}

The converse direction is in general much harder, but is known for special choices of weights. 
\begin{itemize}
    \item If \(\v\) and \(\w\) are constants, this is the uniform YTD conjecture for cscK metrics on toric manifolds. If \(\v\) is constant and \(\w\) is affine, this is the uniform YTD conjecture for extremal metrics on toric manifolds. 
    Both these conjectures were proved recently \cite{TH,Leg19,ApoLN,Li,LLS21,NS} thanks to the breakthrough of Chen--Cheng \cite{CCIII}, its adaptation by He to the extremal setting \cite{He}, and earlier works, notably \cite{SKD,ZZ}.
    
    \item If only \(\v\) is constant, the converse of Theorem~\ref{cscK-to-stability} is known for all \(\w \in \mathcal{C}^{\infty}(P,\mathbb{R})\) by \cite{LLS21}. 
    
    \item For \(\v\)-solitons on Fano toric manifolds, which correspond to choosing an arbitrary weight \(\v \in \mathcal{C}^{\infty}(P,\mathbb{R}_{>0})\) and \(\w(x)= 2(\ell\v(x) + d_x\v(x))\) (see \cite[Proposition~1]{AJL}), it was proved first in \cite{BB} that the converse of Theorem~\ref{cscK-to-stability} holds for general weight \(\v\), and much earlier in \cite{WZ04} for the weight corresponding to Kähler-Ricci solitons. We note that \cite{LH} proved that the general uniform YTD conjecture holds for \(\v\)-solitons on general Fano manifolds, and refer to Section~\ref{section-soliton} for a discussion of \(\v\)-solitons on semisimple principal toric fibrations. 
    
    \item Finally, as we shall explain in details in the next sections, the converse of Theorem~\ref{cscK-to-stability} was proven by the second author for weights corresponding to extremal Kähler metrics on semisimple principal toric fibrations \cite{SJ}. 
\end{itemize}

\subsection{Construction of semisimple principal toric fibration}

In this section we briefly recall the construction of semisimple principal toric fibrations introduced in \cite{ACGTII}. We take the point of view of \cite{AJL}, which generalized the construction when the fiber is not necessarily toric.

Let \(\mathbb{T}\) be an \(\ell\)-dimensional compact torus with Lie algebra $\mathfrak{t}$.
Let $ \pi_B : Q \longrightarrow B$ be a principal \(\mathbb{T}\)-bundle over a $n$-dimensional product of cscK manifold \((B,J_B,\omega_B):=\prod_{a=1}^k (B_a,J_a, \omega_a)\) satisfying the Hodge condition $\omega_a \in H^2(B_a,\mathbb{Z})$.  We supposed that $Q$ is equipped with a principal connection \(\theta  \in \Omega^1(Q) \otimes \mathfrak{t}\) whose curvature satisfies
\[ d\theta = \sum_{a=1}^k \pi_B^*(\omega_a) \otimes p_a, \]

\noindent where  \(p_a\in \mathfrak{t}\) define one-parameter subgroups of \(\mathbb{T}\).   Let \((X,J_X,\omega_X)\) be a toric projective manifold under the action of \(\mathbb{T}\). Since $\T$ acts on various spaces, to avoid confusion, we under-script the space on which $\T$ acts, e.g. $\T_X$ acts on $X$. We consider the $2(\ell+n)$ dimensional smooth manifold

\begin{equation}{\label{define-Y}}
    Y:=(X \times Q)/\mathbb{T},
\end{equation}

\noindent  where the $\mathbb{T}_{X\times Q}$-action on $(X \times Q)$ is given by $ \gamma \cdot  (x, q) = (\gamma \cdot x, \gamma^{-1} \cdot q)$, $x \in X$, $q \in Q$  and $\gamma \in  \mathbb{T}$.
Let $\mathcal{H}:=ann(\theta) \subset  TQ$ be the horizontal distribution on the principal bundle $Q$ with respect to  $\theta$. We consider the smooth section of the endomorphism of $T(X \times Q)$

\begin{equation*}
    J_Y:=J_X \oplus J_B,
\end{equation*}

\noindent where $J_B$ acts on $Q$ via the unique horizontal lift of vector fields on $B$ to $\mathcal{H}$. The section $J_Y$ is invariant with respect to the $\mathbb{T}_{Q\times X}$-action and it is shown in \cite[Section 5]{AJL} that $J_Y$ descends to a complex structure  (still denoted by $J_Y$) on $Y$.

Let $P \subset \mathfrak{t}^*$ be the Delzant polytope associated to  $(X,J_X,\omega_X,\T)$ \cite{TD}. By definition the $\T$-action on $X$ is hamiltonian and we denote by $\mu : X \longrightarrow P $ its moment map. We consider the $2$-form on $X\times Q$

\begin{equation}{\label{compatible-form}}
    \omega_Y:= \omega_X + \sum_{a=1}^k( p_a \circ\mu  + c_a)\pi_B^*(\omega_a) + d\mu \circ \theta,
\end{equation}

\noindent  where the \(k\)-tuple of real constants \((c_a)\) are such that for all \(a\), \(p_a\circ\mu + c_a>0\) and  $d\mu \circ \theta$ is a $2$-form understood as the contraction of the $\mathfrak{t}$-valued one form $\theta$ and the $\mathfrak{t}^*$-valued one form $d\mu$. It is shown in \cite[Section 5]{AJL} that $\omega_Y$ is basic with respect to  $\mathbb{T}_{X \times Q}$ and as such, it is the pullback of a K\"ahler form (still denoted by $\omega_Y$) on $Y$. The $\T_X$-action on $X$ induces an action on $X \times Q$ by the natural $\T_X$-action  on the first factor. This action commutes with  $\T_{X\times Q}$ and therefore descends to a $\T_Y$-action on $Y$. The K\"ahler form $\omega_Y$ on $Y$ defined via the basic $2$-form (\ref{compatible-form}) on $X \times Q$ is $\T_Y$-invariant (see \cite[Section 5]{AJL} for more details).

The K\"ahler metrics $\omega_a$ on $B_a$, the connection form $\theta$ and the constants $p_a \in \mathfrak{t}$ are fixed. The K\"ahler manifold $(Y,J_Y,\omega_Y,\T)$ is a fiber bundle over $B$ with fiber the toric K\"ahler manifold $(X,J_X,\omega_X,\T)$.  Following \cite{ACGT} we define:

\begin{define}{\label{definition-fibration}}
The K\"ahler manifold $(Y,J_Y,\omega_Y,\T)$ defined above is called a \emph{semisimple principal toric fibration}. The $\T$-invariant K\"ahler metric $\omega_Y$ on $Y$ defined from the $\T$-invariant K\"ahler metric $\omega_X$ on $X$ is called a \emph{compatible K\"ahler metric}. A K\"ahler class $\alpha_Y \in H^2(Y,\R)$ containing a compatible K\"ahler metric is called a \emph{compatible K\"ahler class}.
\end{define}

\noindent  In this setup, the constants $c_a$ can vary and they parameterize the compatible K\"ahler classes.

 Let $P^0$ be the interior of $P$ and $\mathring{X}:=\mu^{-1}(P^0)$ be the dense open subset of $X$  of regular orbits for the $\T_X$-action. The $\T$-action on $X$ extends to an effective holomorphic action of the complexified torus $\mathbb{T}^{\mathbb{C}}:= \T \otimes \mathbb{C}$. Fixing any point $x_0 \in \mathring{X}$, we can identify $(\mathring{X},J_X)$ and the orbit $\mathbb{T}^{\mathbb{C}} \cdot x_0 \cong (\mathbb{C}^*)^{\ell}$ of $x_0$:

\begin{equation}{\label{interior-x}}
(\mathring{X},J_X) \cong (\mathbb{C}^*)^{\ell}.
\end{equation}

\noindent Restricting the $\T_{X\times Q}$-action to $\mathring{X} \times Q$, we define

\begin{equation*}{\label{}}
    \mathring{Y}:= (\mathring{X} \times Q)/ \mathbb{T}.  
\end{equation*}

\noindent By (\ref{interior-x}) $\mathring{Y}$ is $\T$-equivariantly biholomorphic to a $(\mathbb{C}^*)^{\ell}$-bundle over $B$. Since $(\mathring{X},J_X)$ compactifies to $(X,J_X)$, the complex principal  $(\mathbb{C}^*)^{\ell}$-bundle $\mathring{Y}$ compactifies to $(Y,J_Y)$.

\begin{remark}{\label{blowdown}}
Semisimple principal toric fibrations $(Y,J_Y,\omega_Y,\T)$ constructed above correspond to semisimple rigid toric fibration introduced and studied in \cite{ACGTI, ACGTII, ACGTIII} when there is no blow-down and the basis $B$ is a global product of cscK Hodge manifolds.
\end{remark}

A particular case, which will be of interest for us, is the case of Fano semisimple principal toric fibrations. 
We recall a characterization of these from \cite{AJL}

\begin{lemma}[{\cite[Lemma~5.11]{AJL}}]
\label{Lemma-AJL}
Assume that each \(B_a\) is a Fano Kähler-Einstein manifold. 
Let \(\omega_a\) denote a Kähler-Einstein metric on \(B_a\) so that \(I_a[\omega_a]= 2\pi c_1(B_a)\), where \(I_a\) denotes the Fano index of \(B_a\). 
We fix a principal bundle with connection \((Q,\theta)\) as before (with associated data \((p_a)\)). 
We further assume that \((X,\omega_X)\) is a Fano toric manifold with a \(\mathbb{T}\)-invariant Kähler form \(\omega_X \in 2\pi c_1(X)\), with the natural choice of moment map \(\mu\). 
If for all \(a\), \(p_a \circ \mu + I_a > 0\), then the semisimple principal fibration \(Y\) associated to the above data is a Fano manifold, and \(\omega_Y\) is in \(2\pi c_1(Y)\) for the \(k\)-tuple \((c_a)=(I_a)\). 
\end{lemma}

Note that in the above situation, the scalar curvature of \(\omega_a\) is indeed constant, equal to \(2n_aI_a\) where \(n_a\) is the complex dimension of \(B_a\). 

\subsection{Projectivization of sum of line bundles as semisimple principal toric fibration}
\label{section-projective-bundles}

We now provide an effective way of constructing examples of semisimple principal toric fibrations. 

Let $(B,\omega_B):=\prod_{a=1}^k(B_a,\omega_a)$ be a product of compact complex manifolds $B_a$ endowed with cscK metrics $\omega_a$ with $[\omega_a]$ primitive element of $H^2(B_a,\mathbb{Z})$. We consider holomorphic line bundles $\mathcal{L}_i \longrightarrow B$, $i=1,\dots, \ell$, and we suppose that their first Chern classes satisfy

\begin{equation*}
    2\pi c_1(\mathcal{L}_i)=\sum_{a=1}^kp_{ai}\pi^*[\omega_a],
\end{equation*}

\noindent  where by definition $p_{ai}$ is the $\omega_a$-degree of $\mathcal{L}_i$. The natural $\mathbb{C}_i^*$-action on $\mathcal{L}_i$ induces an action of $\mathbb{S}_i^1 \subset \mathbb{C}_i^*$, providing a $\mathbb{T}$-bundle $\pi : Q \longrightarrow B$, where $\mathbb{T}=\prod_i \mathbb{S}_i^1$ is a compact $\ell$-torus. We choose a Hermitian metric $h_i$ on $\mathcal{L}_i$ 
and consider the norm function $r_i(u):= (h_i(u,u))^{\frac{1}{2}}$ for any $u$ in $ \mathcal{L}_i$. On $\Tilde{\mathcal{L}}_i$, the $\mathbb{C}^*$-bundle obtained from $\mathcal{L}_i$ by removing the zero section, $r_i$ is positive and we let $t_i=\log(r_i)$. We fix a basis $\boldsymbol{\xi}=(\xi_i)_{i=1}^{\ell}$ of the Lie algebra $\mathfrak{t}$ of $\mathbb{T}$ and we denote by $\xi^{\tilde{\mathcal{L}}_i}$ the generator of the $\mathbb{S}_i^1$-action on $\tilde{\mathcal{L}}_i$. We then consider the  $\mathfrak{t}$-valued one-form $t:=\sum_{i=1}^{\ell} t_i\xi_i$ and we define a connection one-form $\theta$ on $Q$ as the restriction of $d^ct$ to $Q$, seen as the $\T$-bundle of unit element on each $(\mathcal{L}_i,h_i)$. For all $i=1,\dots,\ell$, it satisfies

\begin{equation*}
\theta(\xi^{\tilde{\mathcal{L}}_i})=\xi_i.
\end{equation*}

\noindent We obtain by construction

\begin{equation}{\label{connection}}
\begin{split}
    d\theta &= \sum_{i=1}^{\ell} \xi_i \otimes \pi^ *(\omega_{h_i}) = \sum_{i=1}^{\ell} \xi_i \otimes \big(\sum_{a=1}^k p_{ai}  \pi^ *(\omega_a) \big) \\
        &= \sum_{a=1}^k p_a \otimes \pi^* (\omega_a),
\end{split}        
\end{equation}

\noindent where $\omega_{h_i}$ is the opposite of the curvature form of the Chern connection of $(\mathcal{L}_i,h_i)$ and $p_a=\sum_{i=1}^{\ell} p_{ai} \xi_i$. We consider the $\ell$-projective space $(\P^{\ell}, \omega_ {\P^{\ell}}, \mathbb{T})$ endowed of a toric $\mathbb{T}$-action with respect to a fixed K\"ahler metric $\omega_{\P^{\ell}}$. We fix the principal $\mathbb{T}$-bundle $Q$ with its connection one form $\theta$, the cscK K\"ahler manifolds $(B_a,\omega_a)$ and the toric K\"ahler manifold $(\P^{\ell}, \omega_ {\P^{\ell}}, \mathbb{T})$. From these data, we define a semisimple principal toric fibration $Y:=Q\times_{\T} \P^{\ell}$.  By construction, $Y$ is biholomorphic to the total space of the projective bundle  $\P(E)$, $E:= \mathcal{O} \oplus \bigoplus_{i=1}^{\ell} \mathcal{L}_i$.

Suppose $\omega_{\P^{\ell}}$ belongs to the first Chern class $2\pi c_1(\P^{\ell})$ of $\P^{\ell}$ and denote by $P$ the canonical $\ell$-simplex associate to $(\P^{\ell}, 2 \pi c_1(\P^{\ell}), \mathbb{T}^{\ell})$ via Delzant correspondence \cite{TD}. By (\ref{connection}), any compatible K\"ahler metric on $Y$ is of the form

\begin{equation}{\label{compatible-form-proj-bundle}}
    \omega_Y= \sum_{a=1}^k\bigg(\sum_{i=1}^{\ell}p_{ai}x_i+c_a\bigg)\pi^*(\omega_a) + \omega_{\P^{\ell}}, \text{ } \text{ } \boldsymbol{x}=(x)_{i=1}^{\ell} \in  P
\end{equation}

\noindent with 

\begin{equation}{\label{condition-ca}}
    c_a > \sum_{i=1}^{\ell} p_{ai},
\end{equation}

\noindent In the above formulas,  by abuse of notation, $\omega_{\P^d}$ denotes both the K\"ahler metric on $\P^{\ell}$ and its induced metric in $2\pi c_1(O_Y(\ell+1))$. The tuples $(c_a)$ satisfying \eqref{condition-ca}, parametrize the compatible K\"ahler classes.

Furthermore, suppose that $B$ is a product of K\"ahler-Einstein manifolds $(B,\omega_B):=\prod_{a=1}^k(B_a,\omega_a)$. By Lemma $\ref{Lemma-AJL}$, if we choose  $c_a$ equal to the Fano index $I_a$  of $B_a$, the corresponding compatible K\"ahler form $\omega_Y$ defined in $(\ref{compatible-form-proj-bundle})$ belongs to the first Chern class $2 \pi c_1(Y)$. In particular, if

\begin{equation}{\label{condition-ca2}}
        I_a> \sum_{i=1}^{\ell} p_{ai},
\end{equation}

\noindent $Y$ is a Fano manifold with compatible first Chern class.

\subsection{Extremal metrics on semisimple principal toric fibrations}
\label{section-sptf}

We begin this section by recalling a general characterisation of extremal metrics on compact K\"ahler manifold $(Y,J_Y)$ in a fixed K\"ahler class $\alpha_Y$. By definition a K\"ahler metric $\omega$ is extremal if the symplectic gradient of its scalar curvature is real holomorphic, i.e.

\begin{equation*}
    \mathcal{L}_{\omega^{-1}(Scal(\omega))}J_Y=0.
\end{equation*}

By a well-known result of Calabi \cite{EC} an extremal metric needs to be invariant by a maximal torus $\mathbb{K}$ in the reduced automorphism group $\mathrm{Aut}_{\red}(Y)$ of $Y$. Let us fix such a torus $\K$. For any $\K$-invariant K\"ahler metric $\omega$ on $Y$ the action of $\K$ on $(Y,\omega)$ is hamiltonian and we denote by $\mu_{\omega}$ the corresponding moment map. By a result of Guillemin--Sternberg \cite{VGSS} the image of $\mu_{\omega}$ is a convex compact polytope $P$ in the dual of the Lie algebra $\mathfrak{k}$ of $\mathbb{K}$. For any $\K$-invariant metric $\omega$ in $\alpha_Y$ we normalize  $\mu_{\omega}$ in such way that its image equals $P$. It then follows from \cite{FM} and  \cite[Lemma 1]{AL} that  a $\K$-invariant metric $\omega \in \alpha_Y$ is extremal if and only if

\begin{equation*}
    Scal(\omega)=l_{\ext}(\mu_{\omega}),
\end{equation*}

\noindent where $l_{\ext}$ is the unique affine extremal function such that the Futaki invariant vanishes

\begin{equation}{\label{affine-function}}
    \textbf{F}(l):=\int_Y l(\mu_{\omega}) (Scal(\omega) - l_{ext}(\mu_{\omega}))\frac{\omega^{n}}{n!}=0,
\end{equation}

\noindent for any $l \in \mathrm{Aff}(\mathfrak{k}^*)$.

We now suppose that $(Y,J_Y,\omega_Y,\T)$ is a semisimple principal toric fibration with fiber $(X,J_X,\omega_X,\T)$ and that $\alpha_Y:=[\omega_Y]$ is a compatible K\"ahler class on $Y$. Let $P$ denotes the Delzant polytope associated to $(X,J_X,\omega_X,\T)$. In general, $\T_Y$ is not maximal in $\mathrm{Aut}_{\red}(Y)$. However, by \cite[Proposition 1]{ACGT}, any compatible K\"ahler metric (\ref{compatible-form}) is invariant by a maximal torus $\K_Y \subset \mathrm{Aut}_{\red}(Y)$ containing $\T_Y$ such that the following exact sequence holds

 \begin{equation}{\label{exact-sequence}}
     \{0\} \longrightarrow \mathfrak{t} \longrightarrow \mathfrak{k} \longrightarrow \mathfrak{k}_B \rightarrow \{0\},
 \end{equation}
 
 \noindent where $\mathfrak{t}:=Lie(\T_Y)$, $\mathfrak{k}:=Lie(\K_Y)$ and $\mathfrak{k}_B:=Lie(\K_B)$ where $\K_B \subset \mathrm{Aut}_{\red}(B)$ is a maximal torus such that $\omega_B:= \sum_{a=1}^k\omega_a$ is $\K_B$-invariant (without loss of generality by Lichnerowicz--Matsushima Theorem). 
 
 By construction (\ref{define-Y}), there is an embedding of the space of  $\T$-invariant smooth functions $\mathcal{C}^{\infty}_{\T}(X)$ on $X$ in the space of  $\T$-invariant smooth functions $\mathcal{C}^{\infty}_{\T}(Y)$ on $Y$

 \begin{equation}{\label{embeding-function}}
     \mathcal{C}^{\infty}_{\T}(X) \subset \mathcal{C}^{\infty}_{\T}(Y).
 \end{equation}

\noindent Moreover, by computations in  \cite[p. 380]{ACGTI}, the scalar curvature $Scal(\omega_Y)$  of compatible K\"ahler metric $\omega_Y$ on $Y$ corresponding to a K\"ahler metric $\omega_X$ on $X$ belongs to $ \mathcal{C}^{\infty}_{\T}(X) \subset \mathcal{C}^{\infty}_{\T}(Y)$ and is equal to
 
 \begin{equation}{\label{formula-scalar}}
 \ Scal(\omega_Y)=\sum_{a =1}^k \frac{s_a}{p_a \circ \mu_{\omega_X}  +c_a} + \frac{1}{\mathrm{v}(\mu_{\omega_X})} Scal_{\v}(\omega_X),
\end{equation}

\noindent where $s_a$ are the constant scalar curvatures of $\omega_a$ and $n_a:=\dim(B_a)$, $Scal_{\v}(\omega_X)$ is the $\v$-weighted scalar curvature of $\omega_X$ (see Definition \ref{weighted-cscK}) with respect to the weight $\v(x) := \prod_{a=1}^k (p_a(x) + c_a)^{n_a}$.  It  then follows from $(\ref{exact-sequence})$ and (\ref{formula-scalar}) that (\ref{affine-function}) reduces to an equation on $X$ and is equivalent to
 
 \begin{equation*}
     \int_X(Scal_\v(\omega_X) - \w(\mu_{\omega_X}))l(\mu_{\omega_X})\omega_X^{\ell}=0,
 \end{equation*}

\noindent for every $l \in \mathrm{Aff}(\mathfrak{t}^*)$, where the moment maps $\mu_{\omega_X}$ are normalized such that $\mu_{\omega_X}(X)=P$ and $\w(x) := \left(l_{\ext}(x) - \sum_{a=1}^k \frac{s_a}{p_a(x)+c_a}\right) \v(x)$. To summarize we have the following (see \cite[Section 3.5]{ACGT} or \cite[Lemma 5.14]{AJL} for more details).

\begin{prop}{\label{prop-affine-extremal}}
The affine extremal function $l_{\ext}$ of a compatible K\"ahler class belongs to the image of the map $\mathrm{Aff}(\mathfrak{t^*}) \varhookrightarrow \mathrm{Aff}(\mathfrak{k^*})$ induced by $(\ref{exact-sequence})$.
\end{prop}

\begin{remark}
It follows from Proposition \ref{prop-affine-extremal}, \cite[(27)]{AJL} and \cite[Lemma 15]{AL} that the Futaki invariant  (\ref{affine-function}) restricted to $\mathrm{Aff}(\mathfrak{t^*}) \subset \mathrm{Aff}(\mathfrak{k^*})$  coincides, up to a positive multiplicative constant, to the weighted Donaldson--Futaki invariant  (\ref{define-futaki}) for the weights (\ref{weights}).
\end{remark}

 We now state the main existence result of this section, which is one of the main results of \cite{SJ}.

\begin{theorem}[{\cite[Theorem~3]{SJ}}]
\label{theorem-J}
Let \((Y, J_Y, \omega_Y,\T)\) be a semisimple principal  bundle with K\"ahler toric  fiber $(X,J_X, \omega_X,\T)$ and denote by \(P\) its moment polytope. 
Then there exists an extremal Kähler metric in \([\omega_Y]\) if and only if \(P\) is \((\v,\w)\)-uniformly K-stable, where 

\begin{equation}{\label{weights}}
\begin{split}
\v(x) =& \prod_{a=1}^k (p_a(x) + c_a)^{n_a}\\
\w(x) =& \left(l_{\ext}(x) - \sum_{a=1}^k \frac{s_a}{p_a(x)+c_a}\right) \v(x), \\
\end{split}
\end{equation}

\noindent and \(l_{\ext}\) is the unique affine function such that \eqref{annulation-Futaki} holds for \((\v,\w)\). Equivalently, there exists a $(\v,\w)$-cscK metric in $[\omega_X ]$.
\end{theorem}

\begin{proof}
We only sketch the proof of the direction "$(\v,\w)$-uniform K-stability implies existence of a $(\v,\w)$-weighted cscK metric" which is key in the present paper, and refer to the original paper \cite{SJ} for details and the converse direction. We fix the weights $(\v,\w)$ given by (\ref{weights}). By Proposition \ref{prop-affine-extremal} and (\ref{formula-scalar}), compatible extremal metrics on $Y$ correspond to $(\v,\w)$-cscK metrics on $X$  via (\ref{compatible-form}). Then, the existence of $(\v,\w)$-cscK metric in $[\omega_X]$  implies the existence of a (compatible) extremal metric in $[\omega_Y]$.  The main ingredient is the existence result of cscK metric of Chen--Cheng \cite{XC1, XC2, CCIII}, extended by He \cite{He} to the extremal case:

\begin{theorem}[Chen--Cheng, He \cite{XC1, XC2, CCIII, He}]{\label{chen-cheng}}
Let $\alpha$ be a K\"ahler class on a compact K\"ahler manifold $Y$ and $\K$ be a maximal torus in the reduced automorphism group $\mathrm{Aut}_{\red}(Y)$. Then there exists an extremal metric in $\alpha$ if and only if the $\K$-relative Mabuchi energy $\mathcal{M}^{\K}$ is $\K^{\mathbb{C}}$-coercive, i.e.  there exists $ C>0$ and $D>0$ such that

\begin{equation*}
    \mathcal{M}^{\K}(\varphi) \geq C \inf_{\gamma\in \K^{\mathbb{C}}} d_1(0,\gamma \cdot \varphi)  - D,
\end{equation*}

\noindent for every $\varphi \in \mathring{\mathcal{K}}_{\K}(Y,\tilde{\omega}_0)$.
\end{theorem}

\noindent In the above statement, $\tilde{\omega}_0\in \alpha$ is a fixed $\K$-invariant K\"ahler metric, $\mathcal{M}^{\K}: \mathcal{K}_{\K}(Y,\tilde{\omega}_0) \longrightarrow \R $ is the $\K$-relative Mabuchi energy defined on the space of $\K$-relative K\"ahler potential $\mathcal{K}_{\K}(Y,\tilde{\omega}_0):= \{ \varphi \in \mathcal{C}^{\infty}_{\K}(Y) \text{ } | \text{ } \tilde{\omega}_{\varphi}:= \tilde{\omega}_0+dd^c \varphi >0 \}$, $\mathring{\mathcal{K}}_{\K}(Y,\tilde{\omega}_0) \subset \mathcal{K}_{\K}(Y,\tilde{\omega}_0)$ is the space of normalized potential (with respect to the vanishing of the Aubin-Mabuchi functional), $d_1$ is the Darvas distance \cite{TD} and $\K^{\mathbb{C}}:= \K \otimes \mathbb{C}$ is the complexification of $\K$. Moreover $\K^{\mathbb{C}}$ acts on $\mathring{\mathcal{K}}_{\K}(Y,\tilde{\omega}_0)$ via the natural action on $\K$-invariant K\"ahler metrics in $[\tilde{\omega}_0]$. Originally \cite{He}, the coercivity condition in Theorem \ref{chen-cheng}, was expressed in term of the complexification of a maximal compact connected subgroup of $\mathrm{Aut}_{\red}(Y)$ (not necessarily commutative). As observed in \cite[Section 5]{SJ}, the same arguments as in \cite{CCIII, He} provide this statement.

The proof of Theorem \ref{theorem-J} is divided in two steps:

\begin{enumerate}
    \item  show that the $(\v,\w)$-uniform stability with respect to the weights $(\ref{weights})$ of the polytope $P$ implies the $\T^{\mathbb{C}}$-coercivity of the \textit{weighted Mabuchi energy} $\mathcal{M}_{\v,\w}$  corresponding to  $(X,J_X,[\omega_X], \T)$; 
\item adapt the continuity path (\ref{continuity-path}) involved in the proof of Cheng--Cheng and He to obtain the existence of a compatible extremal metric in $[\omega_Y]$ (or equivalently a $(\v,\w)$-weighted cscK metric in $[\omega_X]$).
\end{enumerate}

\noindent The weighted Mabuchi energy $\mathcal{M}_{\v,\w}$ mentioned above is the one introduced by Lahdili \cite{AL}. 

\textit{Step 1}. The first step is an adaptation of \cite{SKD, ZZ} which established this result in the unweighted case, i.e. when $\v=1$ and $\w=l_{\ext}$. Fix $\omega_0 \in [\omega_X]$ a $\T$-invariant K\"ahler metric. On a toric K\"ahler manifold, a $\T$-invariant K\"ahler metric $\omega \in [\omega_0]$ can be defined via two functions:  a $\T$-invariant $\omega_0$-relative K\"ahler potential $\varphi \in \mathcal{K}_{\T}(X,\omega_0)$ and a symplectic potential $u \in \mathcal{S}(\PL)$, which, by definition, $\mathcal{S}(\PL)$ is the space of smooth strictly convex functions on $P^0$ which satisfy the so-called \textit{Abreu boundary conditions} \cite{MA3, VG}. There is a well-known correspondence between a symplectic and a K\"ahler potential defining the same K\"ahler metric \cite{ACGTI, ACGTII, SKD, VG}. Via this correspondence we can consider the weighted Mabuchi energy $\mathcal{M}_{\v,\w}$ as a functional on $\mathcal{S}(\PL)$. A first step is to show, as in the case $\v=1$ \cite[Proposition 3.3.4]{SKD}, that $\mathcal{M}_{\v,\w}$ extends to the space $\mathcal{CV}^{\infty}(P)$ of smooth convex functions on $P^0$ and continuous on $P$,  and that a symplectic potential $u \in \mathcal{S}(\PL)$ defining a $(\v,\w)$-weighted cscK metric realizes the minimum of $\mathcal{M}_{\v,\w}$ over $\mathcal{CV}^{\infty}(P)$. This is proven in \cite[Proposition 7.7]{SJ}. The idea is then, as in \cite{ZZ}, to compare $\mathcal{M}_{\v,\w}$ and $\mathcal{M}_{\v,\w_0}$, where the weight $\w_0$ is the $\v$-weighted scalar curvature $Scal_{\v}(u_0)$  of (the K\"ahler metric defined by) any fixed symplectic potential $u_0$. Then $u_0$ trivially solves

\begin{equation*}
    Scal_{\v}(u)=\w_0, \text{ } u\in \mathcal{S}(\PL),
\end{equation*}

\noindent In other words, there exists a $(\v,\w_0)$-cscK metric in $[\omega_X]$.  We deduce that $\mathcal{M}_{\v,\w_0}$ is bounded from below on $\mathcal{CV}^{\infty}(P)$ by the discussion above. Consequently, by comparing $\mathcal{M}_{\v,\w}$ and $\mathcal{M}_{\v,\w_0}$ and using our hypothesis, we can show that \cite[Proposition 7.9]{SJ} (or \cite{ZZ} for $\v=1$)

    \begin{equation*}
    \mathcal{M}_{\v,\w}(u) \geq C \|u^*\|_{L_1} - D,
    \end{equation*}

\noindent for any $u \in \mathcal{S}(\PL)$. There $D>0$ and $C>0$ are uniform in $u$ and $\| \cdot \|_{L_1}$ is the $L_1$-norm on $P$ with respect to the Lesbegue measure. Denote by $d_1^X$ the Darvas distance on $\mathcal{K}_{\T}(X,\omega_0)$. For a (normalized) K\"ahler potential $\varphi \in \mathcal{K}_{\T}(X,\omega_0)$ corresponding to (normalized) symplectic potential $u \in \mathcal{S}(\PL)$ via the correspondence described above, we can show that $d^X_1(0,\varphi) \leq A \lVert u \rVert + B$, for some positive constant $A$ and $B$.  We deduce that

\begin{equation*}
    \mathcal{M}_{\v,\w}(\varphi) \geq C' \inf_{\gamma \in \T^{\mathbb{C}}} d^X_1(0,\gamma \cdot \varphi) - D'.
\end{equation*}

\noindent for any $\varphi$ (normalized) $\T$-invariant K\"ahler potential, where $C'$ and $D'$ are positive constant. We refer to \cite[Section 7.6]{SJ} for the precise normalization of symplectic and K\"ahler potentials.

\textit{Step 2}. The proof of the direction "coercivity implies existence" in Theorem \ref{chen-cheng} is based on the resolution of the following continuity path

\begin{equation}{\label{continuity-path}}
    t(Scal(\tilde{\omega}_{\varphi}) - l_{ext}(\mu_{\tilde{\omega}_{\varphi}}))=(1-t)(\Lambda_{\tilde{\omega}_{\varphi}}(\tilde{\chi})-n-\ell),
\end{equation}

 \noindent where $t \in [0,1]$,  $\tilde{\chi}$ is a fixed  K\"ahler metric in $[\tilde{\omega}_0]$, $\dim(Y)=n+\ell$, $\varphi \in \mathcal{K}_{\K}(Y,\tilde{\omega}_0)$ and $\Lambda_{\tilde{\omega}_{\varphi}}(\tilde{\chi})$ is the symplectic trace of $\tilde{\chi}$ with respect to $\tilde{\omega}_{\varphi}$. Chen--Cheng and He showed that there exists $t_0 \in (0,1)$ such that
 
 \begin{equation*}
     \tilde{S}:=\{ t \in [t_0,1] \text{ such that }\exists \varphi_t \in \mathcal{K}_{\K}(Y,\tilde{\omega}_0) \text{ solution of } (\ref{continuity-path})\}
 \end{equation*}
 
 \noindent  is non-empty, open and closed in $[t_0,1]$.

We come back to the case of semisimple principal toric fibration. We fix a $\T$-invariant K\"ahler metric $\omega_0 \in [\omega_X]$ and the induced compatible K\"ahler metric $\tilde{\omega}_0 \in [\omega_Y]$, see (\ref{compatible-form}). It follows from $(\ref{embeding-function})$ and (\ref{exact-sequence}) (see \cite[Lemma 7]{ACGT} for a proof) that

\begin{equation}{\label{embedding-potential}}
    \mathcal{K}_{\T}(X,\omega_0) \subset \mathcal{K}_{\K}(Y,\tilde{\omega}_0),
\end{equation}

\noindent where $\K $ is a maximal torus in $\mathrm{Aut}_{\red}(Y)$ containing $\T$ such that $(\ref{exact-sequence})$ holds.  Moreover, for every $\varphi \in  \mathcal{K}_{\T}(X,\omega_0)$, seen as function on $Y$, the K\"ahler metric $\tilde{\omega}_{\varphi}:=\tilde{\omega}_0+dd^c\varphi$ is a compatible K\"ahler metric in $[\omega_Y]$ in the sense of Definition \ref{definition-fibration}. Also, $\tilde{\omega}_{\varphi}$ on $Y$ is the metric induced by $\omega_{\varphi}$ on $X$ via (\ref{compatible-form}) (see \cite[Lemma 7]{ACGT} or \cite[Lemma 5.5]{AJL}).  Following \cite{ACGT}, we then refer to the image of (\ref{embedding-potential}) as the space of \textit{compatible K\"ahler potentials}. Additionally, (\ref{formula-scalar}), (\ref{embedding-potential}), \cite[(27)]{AJL} and Proposition \ref{prop-affine-extremal} (see also \cite[Lemma 5.10]{AJL}) show that

\begin{equation}{\label{mabuchi-restriction}}
    \mathcal{M}^{\K}|_{\mathcal{K}_{\T}(X,\omega_0)} = \mathcal{M}_{\v,\w}.
\end{equation}

\noindent Let $d_1^Y$ be the Darvas distance on $\mathcal{K}_{\K}(Y,\tilde{\omega}_0)$ \cite{TD}. By \cite[Corolarry 6.5]{AJL} and since $\v>0$,

\begin{equation}{\label{darvas}}
      d^X_1(0,\varphi) \geq A ~ d_1^Y(0,\varphi)
\end{equation}

\noindent  for any $\varphi \in  \mathcal{K}_{\T}(X,\omega_0) \subset \mathcal{K}_{\K}(Y,\tilde{\omega}_0)$, where  $A>0$. By (\ref{formula-scalar}), the restriction of (\ref{continuity-path}) to the space of compatible K\"ahler potentials $ \mathcal{K}_{\T}(X,\omega_0) \subset \mathcal{K}_{\K}(Y,\tilde{\omega}_0)$ is equal to

\begin{equation}{\label{continuity-path2}}
    t(Scal_{\v}(\omega_{\varphi}) - \w(\mu_{{\omega}_{\varphi}}))=(1-t)(\Lambda_{\tilde{\omega}_{\varphi}}(\tilde{\chi})-n-\ell),
\end{equation}

\noindent 
We now want to show that there exists $t_0 \in (0,1)$ such that

\begin{equation*}
    S:=\{ t \in [t_0,1] \text{ such that }\exists \varphi_t \in \mathcal{K}_{\T}(X,\omega_0) \subset \mathcal{K}_{\K}(Y,\tilde{\omega}_0) \text{ solution of } (\ref{continuity-path2})\}
\end{equation*}

\noindent is non-empty, open and closed in $[t_0,1]$. Since any $\T^{\mathbb{C}}$-orbit is included in a $\mathbb{K}^{\mathbb{C}}$-orbit, the $\mathbb{T}^{\mathbb{C}}$-coercivity is stronger than the $\K^{\mathbb{C}}$-coercivity. Then, by (\ref{mabuchi-restriction}), (\ref{darvas}) and \textit{Step 1},
  $\mathcal{M}^{\K}$ is $\K^{\mathbb{C}}$-coercive on the space of compatible potentials $\mathcal{K}_{\T}(X,\omega_0)$ suitably normalized. As observed in \cite[Lemma 6.3]{SJ}, we can choose $\tilde{\chi}$ such that $(\ref{continuity-path2})$ is an equation on $X$ and admits a solution $\varphi_{t_0} \in \mathcal{K}_{\T}(X,\omega_0)$ for some $t_0 \in (0,1)$, showing that $S$ is non-empty. The openness follows from an application of the Implicit Function Theorem \cite[Proposition 6.4]{SJ}. From the \textit{openess}, the closedness of $\mathcal{K}_{\T}(X,\omega_0)$ in  $ \mathcal{K}_{\K}(Y,\tilde{\omega}_0)$ (see \cite[Section 3.4]{SJ}) and \textit{Step 1}, following the proof of Theorem \ref{chen-cheng}, we obtain a sequence of compatible K\"ahler metric $\tilde{\omega}_{\varphi_j}$ such that $\overline{\omega}_j:=\gamma_j^*(\tilde{\omega}_{\varphi_j})$, $\gamma_j \in \K^{\mathbb{C}}$, converge to an extremal metric $\overline{\omega}_1$. The space of (normalized) compatible K\"ahler  potential $\mathring{\mathcal{K}}_{\T}(X,\omega_0)$ is not stable under the action of $\K^{\mathbb{C}}$, then either $\overline{\omega}_j$ or $\overline{\omega}_1$ is compatible in general. However, we can show (see \cite[Proof of Proposition 6.5]{SJ}) that  $\overline{\omega}_1$ is of the form of $(\ref{compatible-form})$, for a possibly different principal connection $\theta$ and K\"ahler metrics $\omega_a$. Moreover, we can argue that the K\"ahler metric $\omega_1 \in [\omega_X]$ defining $\overline{\omega}_1$ via (\ref{compatible-form}) is a weighted $(\v,\w)$-cscK metric for the same weights $(\ref{weights})$, see \cite[Proof of Proposition 6.5]{SJ}.

\end{proof}

Note that condition \eqref{annulation-Futaki} corresponds to the vanishing of the modified Futaki character, and \(l_{\ext}\) encodes the extremal vector field. 
In particular the extremal metric above is cscK if and only if \(l_{\ext}\) is constant. 

\begin{remark}
It is remarkable that the condition depends on the base only through the constants \((s_a)\) and the existence of a principal \(\mathbb{T}\)-bundle with connection with corresponding data \((p_a)\). 
In particular, when we obtain an existence result for extremal Kähler metrics, we usually actually obtain the existence of extremal Kähler metrics over a full deformation family of cscK manifolds. This is exactly this fact which is used in the Proof of Proposition \ref{prop-ex-fano} to obtain an existence condition of extremal metric on  $\mathbb{CP}^2$-bundles over any K\"ahler-Einstein Fano threefold depending only on the cohomology class and the degrees of the line bundles.
\end{remark}

\section{Geometric applications of the sufficient condition}
\label{section-applications}

\subsection{The general statement}

 We use the same notations as in Section~\ref{section-geometric} for semisimple principal toric fibrations and the same notations as in Section~\ref{section-labelled-polytopes} for the decomposition of polytopes. 

\begin{corollary}[of Theorem~\ref{prop-combinatorial}]
\label{condition-extremal-arbitrary-fiber}
The semisimple principal toric fibration \((Y, \omega_Y)\) admits an extremal Kähler metric in \([\omega_Y]\) if there exists an \(x_0\in P^0\) and corresponding cone decomposition \(P=\bigcup_j P_j\) such that for all \(j\) and for all \(x\in P_j\)
\begin{equation*}
    \frac{1}{L_j(x_0) }\bigg( \ell+1 +  \sum_{a=1}^k \frac{n_a p_a(x -x_0)}{p_a(x) + c_a}  \bigg) - \frac{1}{2} \left(l_{\ext}(x) - \sum_{a=1}^k \frac{s_a}{p_a(x)+c_a}\right) \geq 0.
\end{equation*}
\end{corollary}

\begin{proof}
By Theorem~\ref{theorem-J}, the sufficient condition of Theorem~\ref{prop-combinatorial} translates as a sufficient condition of existence of extremal Kähler metrics. 
To obtain the statement above, it suffices to note that for the weight \(\v\) involved, we have 
\[ d_x\v(y) = \left( \sum_{a=1}^k \frac{n_a p_a(y)}{p_a(x) + c_a} \right) \v(x). \]
so that in the condition in Theorem~\ref{prop-combinatorial}, we can factor by \(\v(x)\) which is positive everywhere. 
\end{proof}

\subsection{Fibrations with Fano fiber}

We now turn to the fibrations with Fano fiber, in order to use Corollary~\ref{corollary-monotone}. 
With the same notations as in Section~\ref{section-sptf}, we now assume furthermore that the toric fiber is a Fano manifold, and that the Kähler class \([\omega_X]\) is a multiple of the anticanonical class \( 2 \pi c_1(X)\). 
As a consequence, the moment polytope \(P\) is a dilation of a reflexive lattice polytope. 
This implies that the labelled polytope \((P,\textbf{L})\) corresponding to the lattice polytope \(P\) is monotone, with a preferred point \(x_0\) and \(L_1(x_0)=\cdots=L_d(x_0)=t\). 
Assuming without loss of generality that the (anti-)canonical normalization is used for the moment polytope of the fiber, we may further assume that \(x_0=0\), and \(t= \frac{[\omega]}{2\pi c_1(X)}\). 

\begin{corollary}
\label{corollary-fano-fiber}
The semisimple principal toric fibration \((Y,[\omega_Y])\) with Fano toric fiber admits an extremal Kähler metric in \([\omega_Y]\) if \(\forall x\in P\), 
\begin{equation}{\label{condition-general}}
    2(\ell+\sum_a n_a) +2 + \sum_a \frac{t s_a-2n_ac_a}{p_a(x) + c_a} - t l_{\ext}(x) \geq 0 
\end{equation}
\end{corollary}

Note that \(\ell+\sum_a n_a = \dim(Y)\). 

\begin{proof}
Since all \(L_j(x_0)\) are equal to \(t\), the condition from Corollary~\ref{condition-extremal-arbitrary-fiber} further simplifies to 
\begin{equation*}
    2\ell+ 2 +  \sum_{a=1}^k \frac{2n_a p_a(x)+ t s_a}{p_a(x) + c_a} - t l_{\ext}(x) \geq 0 \text{ } \forall x \in P
\end{equation*}
as for Corollary~\ref{corollary-monotone}. 
Writing \(2n_ap_a(x) = 2n_a(p_a(x)+c_a) -2n_ac_a\) yields the statement. 
\end{proof}

While simple enough, and tractable with numerical optimization techniques, the inequality involved is a polynomial inequality in several variables, whose degree can be equal to the dimension of the basis plus one. 
It is difficult to solve formally, but there is a further reduction that allows to get a simpler condition which can be checked by a finite number of evaluations of polynomial functions. 

\begin{corollary}{\label{cor-vertex}}
Assume furthermore that for all \(a\), \(c_a\geq \frac{ts_a}{2n_a}\). 
Then the semisimple principal toric fibration \((Y,[\omega_Y])\) admits an extremal Kähler metric in \([\omega_Y]\) if inequation~\eqref{condition-general} is satisfied at  every vertex of \(P\).
\end{corollary}

\begin{proof}
The inverse of an affine function is convex on the locus where this affine function is positive. 
Hence under the condition in the statement, the function \(\frac{t s_a-2n_ac_a}{p_a + c_a}\) is concave on \(P\). 
Condition~\eqref{condition-general} thus amounts to checking the non-negativity of a concave function on a convex polytope: it is enough to check the non-negativity on vertices. 
\end{proof}

\begin{remark}
In the case of a \emph{simple} principal toric fibration, that is, if there is only one factor in the basis, then the condition becomes extremely simple for classes with  \(c_a\geq \frac{ts_a}{2n_a}\): it is enough to check a degree two polynomial inequation on every vertex of the moment polytope. This is actually used in most of the examples of Section \ref{section-examples}, see e.g. Proposition \ref{prop-ex-fano} or Proposition \ref{prop-rank-one}.
\end{remark}

\begin{remark}
We can write a similar statement for the general case of toric fibrations, by working on the cone decomposition. 
In that case the conditions to impose are: for all \(j\), for all \(a\), \(L_j(x_0)s_a-2n_a(p_a(x_0)+c_a) \leq 0\) and condition~\eqref{condition-extremal-arbitrary-fiber} is satisfied at all vertices of \(P_j\), that is, some vertices of \(P\) and \(x_0\). 
\end{remark}

\subsection{Extremal metrics in the anticanonical class}

An important special case when the toric fiber is Fano is given by the semisimple principal toric fibrations which are themselves Fano. 
By our general sufficient condition, we obtain a very simple condition for the existence of extremal Kähler metrics on Fano toric fibrations. 

\begin{corollary}
A Fano semisimple principal toric fibration \(Y\) admits an extremal Kähler metric in \(c_1(Y)\) if its extremal function \(l_{\ext}\) satisfies: 
\begin{equation}{\label{Fano-condition}}
    \sup l_{\ext} \leq 2(\dim(Y)+1)
\end{equation}
\end{corollary}

\begin{proof}
By Lemma~\ref{Lemma-AJL}, for all \(a\), \(s_a=2n_ac_a\) and the condition from Corollary~\ref{corollary-fano-fiber} becomes 
\begin{equation*}
    2\dim(Y) + 2  - l_{\ext} \geq 0 \text{ on } P
\end{equation*}
\end{proof}

Of course, as in Corollary~\ref{cor-vertex}, it is enough to check this condition on vertices of the polytope. Furthermore, if \(l_{\ext}\) is constant, it is equal to \(2\dim(Y)\) since the class is the anticanonical one. As a consequence, the condition is strictly satisfied: 
\begin{equation}
\label{strictly-satisfied-Fano}
    2\dim(Y) + 2  - l_{\ext} = 2 > 0.
\end{equation}
In particular, we recover that a Fano toric fibration with vanishing Futaki invariant admits a Kähler-Einstein metric :

\begin{prop}[\cite{AJL}]
Let $(Y,\omega_Y)$ be a Fano semisimple principal toric fibration with vanishing Futaki invariant. Then there exists a Kähler-Einstein metric in $2\pi c_1(Y)$.
\end{prop}

More interestingly, we have the following consequence. 

\begin{prop}
Let \((Y,\omega_Y)\) be a Fano semisimple principal fibration. 
Then on a neighborhood of the anticanonical class, a compatible Kähler class admits a cscK metric if and only if its Futaki invariant vanishes. 
\end{prop}

\begin{proof}
By Remark~\ref{remark-continuous}, the non-negativity condition in Theorem~\ref{prop-combinatorial} varies continuously with the weight. 
By equation~\eqref{strictly-satisfied-Fano}, that condition is strictly satisfied at the anticanonical class, so it is satisfied on a neighborhood of this class. 
The only added assumption in Theorem~\ref{prop-combinatorial} translates as vanishing of the Futaki invariant. 
\end{proof}

\begin{remark}
Of course, if the Futaki invariant of the anticanonical class vanishes, this is already known by Lebrun-Simanca \cite{LS93} and \cite{AJL}.  
Similarly, if the anticanonical class is strictly K-unstable, nearby classes will be as well. 
However, in the present setting, working directly with the condition it is not hard to find an explicit neighborhood which works. 
Furthermore, the statement applies even when we do not know whether there exists an extremal Kähler metric or not in the anticanonical class. 
In the current state of knowledge, it could also happen that the anticanonical class is K-semistable (for the notion of relative K-stability adapted to extremal Kähler metrics), and the above proposition would still apply in that case. 
This is a further illustration of a phenomenon observed in \cite{KSSV2}. 
\end{remark}


\subsection{Weighted solitons on Fano semisimple principal toric fibrations}

\label{section-soliton}

Recall a $\v$-soliton is a K\"ahler metric $\omega$ such that

\begin{equation*}
    Ric(\omega) - \omega= \frac{1}{2}dd^clog(\v),
\end{equation*}

\noindent where $Ric(\omega)$ is the Ricci form of $\omega$. On Fano semisimple principal toric fibrations $Y$ with fiber $X$, a K\"ahler metric $\omega_Y \in 2\pi c_1(Y)$ is a $\v$-soliton if and only if its corresponding metric $\omega_X \in 2 \pi c_1(X)$ is $(\v\v_0, \tilde{\v})$-cscK  (see \cite[Lemma 2.2, Lemma  5.11]{AJL}) for the weights $\tilde{\v}:= 2(\ell \v_0(x) \v(x) + d_x (\v_0 \v)(x))$ and $\v_0$ is defined in (\ref{weights}).  Since the polytope must be reflexive hence monotone, one has \(x_0=0\), \(t=1\) and condition~\eqref{condition-monotone} becomes \(\v\geq 0\) on the polytope, which is obviously satisfied. Moreover, by \cite[Proposition 7.8]{SJ}, the $(\v,\tilde{\v})$-uniform K-stability implies the coercivity of the corresponding weighted $(\v,\tilde{\v})$-Mabuchi functional. Thanks to \cite[Theorem 3.5]{LH}, we obtain the existence of a weighted soliton in $2\pi c_1(Y)$. 
This result was obtained in \cite[Theorem 3]{AJL}, by applying K-stability arguments from \cite{LH}. The proof proposed above allows to remain fully on the differential geometric side.  

\begin{corollary}{\cite[Theorem~3]{AJL}}{\label{Fano-soliton}}
Let $Y$ be a Fano semisimple principal toric fibration with associate Delzant polytope $P$. Consider the weighted Donaldson--Futaki invariant $\mathcal{F}$ for the weights corresponding to $\v$-solitons defined above. Then, if $\mathcal{F}$ vanishes, there exists a $\v$-soliton in $2\pi c_1(Y)$. 
\end{corollary}


\section{Examples} 
\label{section-examples}

\subsection{Examples of bases}{\label{exemple-de-base}}

In this section, we comment on examples of possibles bases for the semisimple principal toric fibration construction. 
This allows to determine possible values of \(s_a\) to plug into the condition. 
The easiest way to get a cscK basis is to choose a Kähler-Einstein manifold, equipped with a multiple of its first Chern class when it is definite, and with an arbitrary Kähler class for Calabi-Yau manifolds. 

For canonically polarized manifolds, there always exists a Kähler-Einstein metric in \(- 2\pi c_1(X)\), and there exists such manifolds in every dimension. 
In particular, the value \(s_a = -\frac{2n_a}{k_a} \) are always allowed, for \(k_a\in \mathbb{Z}_{>0}\).
For manifolds with zero first Chern class, there always exist Kähler-Einstein metrics with zero scalar curvature.
For the positive curvature case, since the projective space of dimension \(n\) is a Kähler-Einstein manifold of index \(n+1\), all the values \(s_a=2\frac{n_a(n_a+1)}{k_a}\) are allowed, for \(k_a\in \mathbb{Z}_{>0}\). 
More generally, for a Kähler-Einstein Fano basis of dimension \(n_a\) and index \(I_a\), then all the values \(s_a=2\frac{n_a I_a}{k_a}\) are allowed, for \(k_a\in \mathbb{Z}_{>0}\). 
Note that the Fano index of an \(n\)-dimensional Fano manifold is always an integer between \(1\) and \(n+1\).
Here are a couple known results on existence of Fano Kähler-Einstein manifolds when \(n\) is small or \(I\) is large: 
\begin{itemize}
    \item if \(I=n+1\) then \(X=\mathbb{P}^n\) is the \(n\)-dimensional projective space, and it is Kähler-Einstein,
    \item if \(I=n\) then \(X=Q^n\) is the \(n\)-dimensional quadric, and it is Kähler-Einstein,
    \item if \(n=1\) then \(X=\mathbb{P}^1\), \(I=2\) and it is Kähler-Einstein, 
    \item if \(n=2\) then \(\mathbb{P}^2\) (index 3), \(X=\mathbb{P}^1\times \mathbb{P}^1\) (index 2) and the blowups of \(\mathbb{P}^2\) (index 1) at three or more points are Kähler-Einstein, 
    \item if \(n=3\), then the existence of Kähler-Einstein metrics on a general member of a deformation family of smooth Fano threefolds was recently settled in \cite{ACC+}, and the families where the general member is \emph{not} Kähler-Einstein are the following, in the labelling used in \cite{ACC+}, 
    2.23, 2.26, 2.28, 2.30, 2.31, 2.33, 2.35, 2.36, 3.14, 3.16, 3.18, 3.21, 3.22, 3.23, 3.24, 3.26, 3.28, 3.29, 3.30, 3.31, 4.5, 4.8, 4.9, 4.10, 4.11, 4.12, 5.2.
\end{itemize}

\subsection{\(\mathbb{P}^2\)-fiber over Fano threefold}{\label{section-fano}}

We consider the $2$-dimensional projective space $(\P^2,\mathbb{T}^2,2\pi c_1(\P^2))$. Identifying the lattice  $\Lambda$ of $\mathbb{T}^2$ with $\Z^2$, we consider its labelled moment polytope $(P,\textbf{L})$ in $\R^2$

\begin{equation}{\label{definition-polytope}}
    P=\{ (x_1,x_2)=:x \in \R^2 \text{ } | \text{ }  L_1(x) \geq 0, \text{ } L_2(x) \geq  0, \text{ } L_3(x) \geq 0\},
\end{equation}

\noindent where $L_1(x):= x_1+1$, $L_2(x):=x_2+1$, $L_3(x):=-x_1-x_2 +1$.  Let $(B,\omega_B)$ be a KE Fano threefold with $\alpha_B:=[\omega_B]$ primitive element of $H^2(B,\Z)$ proportional to the first Chern class $2 \pi c_1(B)$. Let
$\mathcal{L}_i \longrightarrow B$ be a holomorphic line bundle of degree $p_i$ proportional to the anticanonical line bundle $-K_B$, i.e. $ p_i \alpha_B= 2\pi c_1(\mathcal{L}_i)$. We consider a \textit{simple} principal toric fibration (i.e. the basis has only one factor) $ \pi : Y:=\mathbb{P}(\mathcal{L}_0 \oplus \mathcal{L}_1 \oplus \mathcal{L}_2) \longrightarrow B$.   Since the holomorphic class of $Y$ is invariant by tensoring $\mathcal{L}_0 \oplus \mathcal{L}_1 \oplus \mathcal{L}_2$ with a line bundle, we can suppose without loss of generality that $L_0=\mathcal{O}$ is the trivial line bundle and $p_i \geq 0$, $i=1,2$. 
When $B$ is a local K\"ahler product of nonnegative cscK metric and $p_1=p_2 > 0$ or $p_2>p_1=0$, it is known \cite[Proposition 11]{ACGTIII}, that there exists an extremal metric in every compatible K\"ahler classes. We then suppose $p_2\geq p_1 > 0$.

The compatible K\"ahler classes are parametrized by constants $c$ and are of the form 

\begin{equation}{\label{alphac}}
    \alpha_c :=  2\pi c_1(O_Y(3)) + c \pi^*(\alpha_B),
\end{equation}

\noindent As introduced in Section \ref{exemple-de-base}, since \(B\) is a Fano threefold, the only possible Fano indices $I$ are $1$, $2$, $3$ or $4$. In the case where $B$ is the quadric $Q_3$ or the projective space $\P^3$ (i.e. if $I=3$ or $I=4$ respectively), Leray-Hirch Theorem shows that $H^2(Y,\R)\cong \R^2$.
It follows that, up to scaling, all K\"ahler classes are compatible, i.e. of the form of (\ref{alphac}). It is known that \cite[Theorem 4]{ACGT}  for $c $ sufficiently large, the class $\alpha_c$ is extremal. The following Proposition gives a precise value for $c$, depending on $p_1$ and $p_2$, from which $\alpha_c$ admits an extremal metric.

\begin{prop}{\label{prop-ex-fano}}
Let $Y=\P(\mathcal{O} \oplus \mathcal{L}_1 \oplus \mathcal{L}_2) \longrightarrow B$ be a simple principal toric fibration over a K\"ahler-Einstein Fano threefold $B$, where $\mathcal{L}_1$ and $\mathcal{L}_2$ are holomorphic line bundles of degrees $1 \leq p_1 \leq p_2 $ proportional to the anti-canonical line bundle $-K_B$ of $B$. Then there exists an extremal metric in $\alpha_c$ for $c\geq 7 p_2$.
\end{prop}

\begin{proof}
Since the arguments are identical for each Fano index $I$, we give the proof only for $I=4$.

By Corollary \ref{cor-vertex}, for $c \geq 4$,  it is sufficient to check (\ref{condition-general}) evaluated in each vertex $v_1:=(-1,2)$, $v_2:=(-1,-1)$, $v_3:=(2,-1)$ of the polytope $P$.

Using Program \ref{P2-program} in Appendix $\ref{appendix}$, we find that the LHS of $(\ref{condition-general})$ evaluated in $v_1$ is a rational fraction in the variables $c$, $p_1$, $p_2$:

\begin{equation*}
    \text{ LHS of } (\ref{condition-general}) = \frac{P(c,p_1,p_2)}{Q(c,p_1,p_2)}.
\end{equation*}

\noindent We give the explicit expression of the polynomials $P$ and $Q$ in Appendix \ref{appendixB}. Suppose now $c \geq 7 p_2 $ and $p_2 \geq p_1 \geq 1$. Then we can find two polynomials

\begin{equation*}
\begin{split}
    R(c):=&12250c^{10} - 73500c^9 - 295470c^8 + 1296540c^7 - 3657150c^6+ 3776220c^5 \\
    & - 6537672c^4 + 5624964c^3 - 6193584c^2 + 85920232c - 1889568
\end{split}    
\end{equation*}

\noindent and

\begin{equation*}
\begin{split}
S(c):=& 6125c^{10} + 18375c^9 + 6615c^8 + 19845c^7 + 127575c^6 \\
&+  382725c^5 + 17496c^4 + 52488c^3 - 288684c^2 - 866052c
\end{split}
\end{equation*}

 \noindent such that

\begin{equation*}
    0 < R(c) \leq P(c,p_1,p_2) 
\end{equation*}

\noindent and

\begin{equation*}
    0< S(c) \text{ } \text{ and } S(c) \geq Q(c,p_1,p_2).
\end{equation*}

\noindent It implies that

\begin{equation*}
 \text{ LHS of } (\ref{condition-general}) = \frac{P(c,p_1,p_2)}{Q(c,p_1,p_2)} \geq \frac{R(c)}{S(c)} \geq 0.
\end{equation*}

\noindent We proceed analogously for the vertex $v_2$ and $v_3$. We conclude the proof by involving Corollary \ref{cor-vertex}.

\end{proof}

\begin{remark}
In Proposition \ref{prop-ex-fano}, we obtain a lower bound on $c$  depending only on the degrees $p_1$ and $p_2$ of the line bundles $\mathcal{L}_1$ and $\mathcal{L}_2$.  For given values of  $p_1$ and $p_2$ it is possible to obtain a more optimal result. Indeed, suppose $p_1$ and $p_2$ are fixed. Then, the LHS of (\ref{condition-general}) is a rational fraction $F$ depending only on the variable $c$. We then only need to look for constant $\alpha$ such that $F$ is non-negative for $c \geq \alpha$. For example, if $B=\P^3$, respectively  $B=Q_3$, $p_1=1$ and $p_2=2$, (\ref{condition-general}) show the existence of an extremal metric in $\alpha_c$ for $c \geq 7.09$, respectively $c \geq 9.08$. We refer to Appendix \ref{appendix} for further examples of application of the sufficient condition on simple principal $\P^2$-fibrations.
\end{remark}

\subsection{Comments on the rank one case}

\subsubsection{Varying \(x_0\) and prescribing weighted scalar curvature on \(\mathbb{P}^1\)}
\label{section-varying}

As noted in Remark~\ref{remark-varying}, it can be useful to vary the base point \(x_0\). 
In this short paragraph, we want to illustrate this phenomenon in the simplest possible case, that is, when working on the one-dimensional polytope \([-1,1]\subset \mathbb{R}\) with the weights \(\v \equiv 1\) and arbitrary \(\w\). 
We further choose the lattice labelling of \([-1,1]\) induced by the lattice \(\mathbb{Z}\subset \mathbb{R}\) (in other word, we work on the anticanonical moment polytope of \(\mathbb{P}^1\)).
More precisely, the labelling \((L_1,L_2)\) is given by \(L_1(x)=1+x\) and \(L_2(x)=1-x\). 
Since \(\v \equiv 1\), we have \(d_x\v \equiv 0\), hence condition (\ref{combinatorial1}) from Theorem~\ref{prop-combinatorial} translates as 
\[ \frac{1}{4}\w|_{[-1,x_0]} \leq \frac{1}{1+x_0} \qquad \text{and} \qquad \frac{1}{4}\w|_{[x_0,1]} \leq \frac{1}{1-x_0} \]
The latter condition is illustrated in Figure~\ref{figure-varying-x_0}, and it is obviously less restrictive if one can choose \(x_0\) than the uniform condition corresponding to the obvious choice of \(x_0=0\) for the monotone lattice polytope \([-1,1]\). 

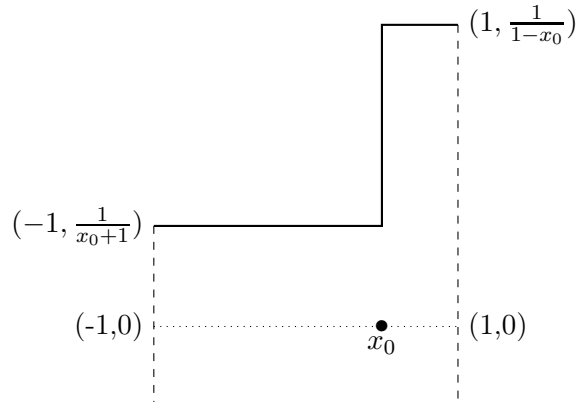
\begin{figure}
\caption{Varying \(x_0\)}
\label{figure-varying-x_0}
\begin{tikzpicture}
\draw[thick] (-2,4/3) -- (1,4/3) -- (1,4) -- (2,4);
\draw[dotted] (-2,0) -- (2,0);
\draw (-2,0) node[left]{(-1,0)};
\draw (2,0) node[right]{(1,0)};
\draw (1,0) node{\(\bullet\)};
\draw (1,0) node[below]{\(x_0\)};
\draw (-2,4/3) node[left]{\((-1,\frac{1}{x_0+1})\)};
\draw (2,4) node[right]{\((1,\frac{1}{1-x_0})\)};
\draw[dashed] (-2,4/3) -- (-2,-1);
\draw[dashed] (2,4) -- (2,-1);
\end{tikzpicture}
\end{figure}

We end this paragraph by recalling that \((1,\w)\)-uniform stability of the lattice polytope \([-1,1]\) translates to existence of certain canonical Kähler metrics on \(\mathbb{P}^1\) thanks to \cite{LLS21}. 

\subsubsection{Extremal metrics on \(\mathbb{P}^1\)-bundles}

We have focused on applications of our sufficient condition to semisimple principal toric bundles with dimension two toric fiber. 
This is because in the case of a one-dimensional toric fiber, quite a few strong results have been shown in \cite{ACGTIII}. 
For example, it is proved in \cite[Proposition~11]{ACGTIII} that if all factors \((B_a,\omega_a)\) of the basis have non-negative constant scalar curvature, and the fiber is one-dimensional, then there exists an extremal Kähler metric in all compatible Kähler classes. 

There cannot be such a result if some factors of the basis have negative constant scalar curvature, as shown by examples in \cite{ACGTIII}. 
More importantly, some of these examples motivated the initial introduction of the notion of uniform K-stability, as they are likely relatively K-polystable in the sense of \cite{Sze07}, but do not admit extremal Kähler metrics. 

On the positive side, by \cite[Theorem~1]{ACGTIII}, there always exist extremal Kähler metrics on a semisimple principal \(\mathbb{P}^1\)-fibration, when all the \(c_a\) are large enough, an example of existence of extremal Kähler metrics in an adiabatic regime for fibrations. 
However, it is not so easy to derive explicit Kähler classes with extremal Kähler metrics from this asymptotic proof. 
A possible approach to get explicit classes with extremal Kähler metrics would be to compute the extremal polynomial (in the terminology of \cite{ACGTIII}) and check when it is positive. 
This is less practical than our sufficient condition, which involves only checking the positivity of a polynomial at two points. 
We provide in the appendix an elementary computer program which checks whether our sufficient condition is satisfied for a simple principal \(\mathbb{P}^1\)-fibration, which could easily be adapted to the case of a semisimple principal \(\mathbb{P}^1\)-fibration. 

\subsubsection{A more explicit example}

Consider \(B\) a three-dimensional canonically polarized manifold, equipped with its Kähler-Einstein metric in \(- 2 \pi c_1(X)\), whose scalar curvature is thus equal to \(-6\). 
We consider the sufficient condition for existence of extremal Kähler metrics in admissible Kähler classes on the \(\mathbb{P}^1\)-bundles \(\mathbb{P}(\mathcal{O}_B\oplus K_B^m)\). 
Up to rescaling and symmetry, this amounts to checking \((\v,\w)\)-uniform K-stability of the reflexive lattice polytope \([-1,1]\subset \mathbb{R}\) with respect to the weights 
\[ \v(x)=(px+c)^3 \qquad \text{and} \qquad \w(x)=\left(l_{\ext}(x)-\frac{-6}{px+c}\right)(px+c)^3 \]
where \(p\in \mathbb{Q}\), \(c\in \mathbb{R}\) and \(c>p>0\).
Our sufficient condition allows to obtain the following explicit families of extremal Kähler classes. 
We only show an example with very rough estimates to illustrate the results, but of course one could get much more classes by using more precise estimates in the proof, and even more classes by using the sufficient condition in Theorem~\ref{prop-combinatorial} in its full generality. 

\begin{prop}{\label{prop-rank-one}}
\label{prop-negative-base}
With the above notations, if \(c\geq 15p\), then \([-1,1]\) is \((\v,\w)\)-uniformly K-stable. The corresponding Kähler classes on the \(\mathbb{P}^1\)-bundles \(\mathbb{P}(\mathcal{O}\oplus K_B^m)\) admit extremal Kähler metrics. 
\end{prop}

\begin{proof}
Using Program~\ref{rank-one-program} in the appendix or straightforward but tedious computations, we obtain up to elementary simplifications that the sufficient condition reads as 
\[ 75c^7-300c^6-65c^5p^2+160c^4p^2-15c^3p^4-180c^2p^4-27cp^6+48p^6 \]
is greater than 
\[\lvert -75c^6p+5c^4p^3+80c^3p^3-105c^2p^5+15p^7 \rvert \]Without attempting to give an optimal result, we may as well check that it is greater than 
\[ 75c^6p+5c^4p^3+80c^3p^3+105c^2p^5+15p^7 \] 
since \(c\) and \(p\) are positive. 
Writing \(c=\alpha p\) for some \(\alpha >1\) and simplifying by \(p^6\), we get a linear inequation in \(p\)
\begin{equation}
\label{inequation-linear-rank-one}
pA+B \geq 0
\end{equation}
where 
\begin{align*}
    A & = 75\alpha^7-75\alpha^6-65\alpha^5-5\alpha^4-15\alpha^3-105\alpha^2-27\alpha-15 \\
    B & = -300\alpha^6+160\alpha^4-80\alpha^3-180\alpha^2+48 
\end{align*}
Since \(\alpha >1\), the coefficient \(A\) is larger than \( (75\alpha - 307)\alpha^6 \) and in particular, it is non-negative for \(\alpha \geq \frac{307}{75}\). 
Using the same lower bound for the leading coefficient, inequation~\eqref{inequation-linear-rank-one} is certainly satisfied at \(p=1\) if 
\[ (75\alpha -307)\alpha^6 -300\alpha^6-160\alpha^4-80\alpha^3-180\alpha^2-48 \geq 0 \]
Using again \(\alpha> 1\) and very rough estimates, this is implied by the inequality 
\[ (75\alpha-1075)\alpha^6 \geq 0 \]
The latter is satisfied at least for \(\alpha \geq 15\), and since \(15\geq \frac{307}{75}\), we obtain that if \(\alpha \geq 15\), the sufficient condition is satisfied for all \(p\geq 1\). 
\end{proof}

\appendix

\section{An elementary Python program}{\label{appendix}}
We provide, as a courtesy to the reader, elementary Python programs using SymPy that checks the sufficient condition from Corollary~\ref{cor-vertex} for \emph{simple} principal toric fibrations (that is, the basis has only one factor) with Fano toric fiber \(X\) of dimension one or two such that \([\omega_X]\) a multiple of \( 2 \pi c_1(X)\). 

The only data from the simple principal toric bundle needed to compute the condition is: 
\begin{itemize}
    \item from the basis, the dimension \(n\in \mathbb{Z}\) and scalar curvature \(s\in \mathbb{Q}\)
    \item from the Fano toric fiber of dimension \(\ell\in \{1,2\}\), the reflexive moment polytope \(P\subset \mathbb{R}^{\ell}=\mathbb{Z}^{\ell}\otimes \mathbb{R}\), and the multiple \(t=\frac{  [\omega_X]}{2 \pi c_1(X)}\in \mathbb{R}\)
    \item the one-parameter subgroup \(p\) from the principal bundle, identified with an integer \(p\in \mathbb{Z}\) if the fiber is one-dimensional, and with an element \(p=(p_1,p_2)\in \mathbb{Z}^2\) if the fiber is of dimension two,
    \item and the constant \(c\in \mathbb{R}\) defining the admissible Kähler class. 
\end{itemize}
We wish to compute the expression given by the right-hand side of~\eqref{condition-general}
\[ \test = 2(\ell+n+1)+\frac{ts-2nc}{p(x)+c}-tl_{\ext}(x) \]
in order to check the condition. 
For this, it suffices to compute the extremal function \(l_{\ext}\) by solving the linear system which defines it. 
Our short programs compute \(l_{\ext}\), then \(\test\), then evaluate \(\test\) at the vertices of \(P\) and returns the minimum if all data are explicitly given. 
It the minimum returned by the program is non-negative, the data correspond to a simple principal toric fibration with an admissible Kähler class and \(c>\frac{ts}{2n}\), then there exists an extremal Kähler metric. 
We may also let some of the data remain unknown and treat them as variables.

\begin{lstlisting}[language=Python, caption=Rank one simple principal toric fibrations, label=rank-one-program]
import sympy as sym
# variable on the line (here the fiber is one-dimensional)
x = sym.symbols('x')
# data of the simple principal toric fibration
p, c = sym.symbols('p,c')
n, s, t = 3, -6, 1
# weights
l = c+p*x
v, w0 = l**n, -s*l**(n-1) # for now, unknown l_ext is replaced with zero
# Donaldson-Futaki invariant with weights (v,w0)
def DF0(f):
    interior=sym.integrate(f*w0, (x, -t, t))
    facets=(f*v).subs(x,-t)+(f*v).subs(x,t)
    return(interior+facets)
# Compute the extremal function lext
X=sym.Matrix(2, 1, [1, x])
M=sym.Matrix(2, 2, lambda i,j: 
    sym.integrate(X[i,0]*X[j,0]*v, (x, -t, t)))
V=sym.Matrix(2, 1, [DF0(1), DF0(x)])
Lext=M.LUsolve(V)
lext=((Lext.T)*X)[0,0]
# Compute expression test at the two vertices and print it
test=2*(1+1+n)+(t*s-2*n*c)/l-t*lext
print(sym.factor(test.subs(x,-t)))
print(sym.factor(test.subs(x,t)))
\end{lstlisting}

Program~\ref{rank-one-program} prints the condition to check when \(c\) and \(p\) are variables, \(n=3\), \(s=-6\) and \(t=1\), as used in Proposition~\ref{prop-negative-base}. 
By modifying Line 5 and 6, one can obtain the conditions for an arbitrary simple principal \(\mathbb{P}^1\)-bundle. 

\begin{lstlisting}[language=Python, caption= Simple principal \(\mathbb{P}^2\) toric fibrations, label=P2-program]
import sympy as sym
# variables on the plane
x1, x2 = sym.symbols('x1,x2')
# data of toric fibration and admissible Kahler class
c, p1, p2, n, s, t = 12, 1, 2, 3, 18, 1 
## weights associated to the data
l=c+p1*x1+p2*x2
v=l**n
w0=-s*l**(n-1) # for now, unknown l_ext replaced with zero
# list of vertices of the polytope
vert= [[2*t,-t], [-t,-t], [-t,2*t]]
# Donaldson-Futaki invariant with weights (v,w0)
def DF0(f):
  interior=sym.integrate(sym.integrate(f*w0,(x2,-t,t-x1)),(x1,-t,2*t))
  facet1=sym.integrate((2*f*v).subs(x2,-t),(x1,-t,2*t))
  facet2=sym.integrate((2*f*v).subs(x2,t-x1),(x1,-t,2*t))
  facet3=sym.integrate((2*f*v).subs(x1,-t),(x2,-t,2*t))
  return(interior+facet1+facet2+facet3)
# Compute the extremal function l_ext
X=sym.Matrix(3, 1, [1, x1, x2])
M=sym.Matrix(3, 3, lambda i,j:
    sym.integrate(sym.integrate(X[i,0]*X[j,0]*v,(x2,-t,t-x1)),(x1,-t,2*t)))
V=sym.Matrix(3, 1, [DF0(1), DF0(x1), DF0(x2)])
Lext=M.LUsolve(V)
lext=((Lext.T)*X)[0,0]
# Compute and print the minimum of expression test on vertices
test=2*(1+2+n)+(t*s-2*n*c)/l-t*lext
test_vertices=test.subs(x1,vert[0][0]).subs(x2,vert[0][1])
for i in range(1,len(vert)):
  test_vertices=sym.Min(test_vertices,
                        test.subs(x1,vert[i][0]).subs(x2,vert[i][1]))
print("The minimum of expression test on vertices is ", test_vertices)
\end{lstlisting}

Program~\ref{P2-program} computes the condition when all the data are given the fixed values \((c,p1,p2,n,s,t)=(12,1,2,3,18,1)\). 
Changing the values on the right-hand side of Line 5 allows to check the sufficient condition for arbitrary fixed values. 
If one wants one or several of the above quantities to be treated as variables, for example \(c\), \(p_1\) and \(p_2\), it suffices to remove these and the corresponding values on the right in Line 5 and add the line 
\begin{lstlisting}[language=Python, firstnumber=6]
c, p1, p2 = sym.symbols('c,p1,p2')
\end{lstlisting}
Since the program will now compute values of \(\test\) as symbolic expressions, it will no longer be able to determine the minimum. 
One should thus replace Lines 28--32 for example by 
\begin{lstlisting}[language=Python, firstnumber=28]
print(sym.separatevars(test.subs(x1,vert[2][0]).subs(x2,vert[2][1])))
\end{lstlisting}
to get the expressions from appendix~\ref{appendixB}, to be used in the proof of Proposition~\ref{prop-ex-fano}. 

Similarly, it is very easy to modify the program to consider another Fano toric surface as fiber (Recall that there are five smooth Fano toric surfaces: \(\mathbb{P}^1\times\mathbb{P}^1\) and the blowups of \(\mathbb{P}^2\) at up to three fixed points under the torus action). 
It suffices to modify Lines 10--18 according to the desired polytope. 
For example, if one wants to work with fiber the first Hirzebruch surface (i.e. the blowup of \(\mathbb{P}^2\) at one point), then it suffices to replace Lines 10--18 with 
\begin{lstlisting}[language=Python, firstnumber=10]
# list of vertices of the polytope
vert= [[-t,-t], [t,-t], [t,0], [-t,2t]]
# Donaldson-Futaki invariant with weights (v,w0)
def DF0(f):
  interior=sym.integrate(sym.integrate(f, (x2, -t, t-x1)), (x1, -t, t))
  facet1=sym.integrate(f.subs(x2,-t), (x1, -t, t))
  facet2=sym.integrate(f.subs(x2,t-x1), (x1, -t, t))
  facet3=sym.integrate(f.subs(x1,-t), (x2, -t, 2t))
  facet4=sym.integrate(f.subs(x1,t), (x2, -t, 0))
  return(interior+facet1+facet2+facet3+facet4)
\end{lstlisting}


\section{Complement of proof of Proposition \ref{prop-ex-fano}}{\label{appendixB}}

\begin{equation*}
    \begin{split}
    P(c,p_1,p_2):=&12250c^{10} + 24500c^9p_1 - 39690c^8p_1^2 + 18060c^7p_1^3 - 22470c^6p_1^4 - 31752c^5p_1^5 \\
    &- 53376c^4p_1^6+ 22740c^3p_1^7 - 57024c^2p_1^8 + 1312p_1^10 - 49000c^9p_2  \\
    &+ 34650c^7p_1^2p_2 + 286860c^6p_1^3p_2 + 152460c^5p_1^4p_2 + 360972c^4p_1^5p_2 \\
    &- 59520c^3p_1^6p_2 + 230112c^2p_1^7p_2 + 18288cp_1^8p_2 - 464p_1^9p_2 - 127890c^8p_2^2 \\
    &- 212310c^7p_1p_2^2 - 615510c^6p_1^2p_2^2 - 373212c^5p_1^3p_2^2 - 921924c^4p_1^4p_2^2 \\
    &- 425376c^2p_1^6p_2^2 - 160632cp_1^7p_2^2 - 19296p_1^8p_2^2 + 141540c^7p_2^3 + 657300c^6p_1p_2^3 \\
    &+ 603288c^5p_1^2p_2^3 + 1408632c^4p_1^3p_2^3 + 390936c^3p_1^4p_2^3 + 571536c^2p_1^5p_2^3 \\
    &+ 349440cp_1^6p_2^3 + 41376p_1^7p_2^3 - 328650c^6p_2^4 - 531720c^5p_1p_2^4 - 1421136c^4p_1^2p_2^4 \\
    &- 806100c^3p_1^3p_2^4 - 829080c^2p_1^4p_2^4 - 497592cp_1^5p_2^4 - 22416p_1^6p_2^4 + 212688c^5p_2^5 \\
    &- 43812c^3p_1^5p_2^2 + 860184c^4p_1p_2^5 + 849456c^3p_1^2p_2^5 + 906192c^2p_1^3p_2^5 \\
    & + 485712cp_1^4p_2^5 - 7488p_1^5p_2^5 - 286728c^4p_2^6 - 527016c^3p_1p_2^6 - 725760c^2p_1^2p_2^6 \\
    &- 329952cp_1^3p_2^6 + 127890c^8p_1p_2  + 7352cp_1^9 + 22656p_1^4p_2^6 + 150576c^3p_2^7  \\
    &+ 363168c^2p_1p_2^7 + 156096cp_1^2p_2^7 - 25728p_1^3p_2^7 - 90792c^2p_2^8 - 46368cp_1p_2^8  \\
    &+ 16992p_1^2p_2^8 + 10304cp_2^9 - 7040p_1p_2^9 + 1408p_2^10 + 132300c^7p_1^2 + 105840c^6p_1^3 \\
    &- 11340c^5p_1^4 + 125496c^4p_1^5 + 151200c^3p_1^6 - 79056c^2p_1^7 + 60048cp_1^8  \\
    &- 12096p_1^9 - 396900c^7p_1p_2 - 449820c^6p_1^2p_2 - 260820c^5p_1^3p_2 - 374220c^4p_1^4p_2 \\
    &- 420336c^3p_1^5p_2 + 358992c^2p_1^6p_2 - 364176cp_1^7p_2 + 17712p_1^8p_2 + 396900c^7p_2^2  \\
    &+ 714420c^6p_1p_2^2 + 601020c^5p_1^2p_2^2 + 378756c^4p_1^3p_2^2 + 743904c^3p_1^4p_2^2  \\
    &- 557280c^2p_1^5p_2^2 + 734832cp_1^6p_2^2 + 84240p_1^7p_2^2 - 476280c^6p_2^3 - 680400c^5p_1p_2^3 \\
    &- 282744c^4p_1^2p_2^3 - 728784c^3p_1^3p_2^3 + 99792c^2p_1^4p_2^3 - 1073520cp_1^5p_2^3  \\
    &- 287280p_1^6p_2^3 + 340200c^5p_2^4 + 45360c^4p_1p_2^4 + 568512c^3p_1^2p_2^4 + 829440c^2p_1^3p_2^4  \\
    &+ 1551312cp_1^4p_2^4 - 244944c^3p_1p_2^5- 241056c^2p_2^7 - 736128cp_1p_2^7 -18144c^4p_2^5 \\
    &+ 415152p_1^5p_2^4 - 279072p_1^2p_2^7 + 184032cp_2^8 + 139968p_1p_2^8 - 31104p_2^9  \\
    &   - 1175472c^2p_1^2p_2^5 - 1732752cp_1^3p_2^5 - 358992p_1^4p_2^5 + 81648c^3p_2^6\\
    &+ 843696c^2p_1p_2^6 + 1436400cp_1^2p_2^6 + 323568p_1^3p_2^6 
    \end{split}
\end{equation*}

\begin{equation*}
    \begin{split}
Q(c,p_1,p_2):=&6125c^9 + 2205c^7p_1^2 + 210c^6p_1^3 + 14175c^5p_1^4 - 7812c^4p_1^5 + 24c^3p_1^6 \\
&+ 9072c^2p_1^7 - 5004cp_1^8 + 688p_1^9 - 2205c^7p_1p_2 - 315c^6p_1^2p_2 - 28350c^5p_1^3p_2 \\
&- 31752c^2p_1^6p_2 + 20016cp_1^7p_2 - 3096p_1^8p_2 + 2205c^7p_2^2 - 315c^6p_1p_2^2 + 42525c^5p_1^2p_2^2 \\
&- 7812c^4p_1^3p_2^2 + 4356c^3p_1^4p_2^2 + 40824c^2p_1^5p_2^2 - 40320cp_1^6p_2^2 + 4464p_1^7p_2^2 \\
&+ 210c^6p_2^3 - 28350c^5p_1p_2^3 - 7812c^4p_1^2p_2^3 - 8592c^3p_1^3p_2^3 - 22680c^2p_1^4p_2^3 \\
&+ 50904cp_1^5p_2^3 - 1176p_1^6p_2^3 + 19530c^4p_1^4p_2 - 72c^3p_1^5p_2  \\
&+ 14175c^5p_2^4 + 19530c^4p_1p_2^4 + 4356c^3p_1^2p_2^4 - 22680c^2p_1^3p_2^4 - 56196cp_1^4p_2^4 \\
&- 1224p_1^5p_2^4 - 7812c^4p_2^5 - 72c^3p_1p_2^5 + 40824c^2p_1^2p_2^5 + 50904cp_1^3p_2^5 - 1224p_1^4p_2^5 \\
&+ 24c^3p_2^6 - 31752c^2p_1p_2^6 - 40320cp_1^2p_2^6 - 1176p_1^3p_2^6 + 9072c^2p_2^7 + 20016cp_1p_2^7 \\
&+ 4464p_1^2p_2^7 - 5004cp_2^8 - 3096p_1p_2^8 + 688p_2^9    \end{split}
\end{equation*}


\begin{thebibliography}{}

\bibitem{MA3} M. Abreu, \textit{K\"ahler geometry of toric manifolds in symplectic coordinates}, Fields Institute Comm \textbf{35} (2008) , 1-24.

\bibitem{ApoLN} V. Apostolov, \textit{The K\"ahler geometry of toric manifolds}, Lecture Notes of CIRM winter school, 2019, available on the author's website. 

\bibitem{ACGTI} V. Apostolov, D. Calderbank, and P. Gauduchon, \textit{ Hamiltonian 2-Forms in Kähler Geometry, I General Theory}, J. Differential Geom. \textbf{73}, Number 3 (2006), 359-412.


\bibitem{ACGTII}V. Apostolov, D. M. J. Calderbank, P. Gauduchon and C. Tønnesen-Friedman, 
\textit{Hamiltonian 2-forms in K\"ahler geometry, II Global classification}, J. Differential Geom. \textbf{68} (2004),
277–345.


\bibitem{ACGTIII} V. Apostolov, D. M. J. Calderbank, P. Gauduchon and C. Tønnesen-Friedman, \textit{Hamiltonian 2-forms in Kähler geometry. III: Extremal metrics and stability.} Invent. Math. 173, No. 3, 547-601 (2008).

\bibitem{ACGT} V. Apostolov, D. M. J. Calderbank, P. Gauduchon and C. Tønnesen-Friedman, \textit{Extremal K\"ahler metrics on projective bundles over a curve.} Adv. Math. \textbf{227} (2011), 2385–2424.

\bibitem{AJL} V. Apostolov, S. Jubert, A. Lahdili, \textit{Weighted K-stability and coercivity with applications to extremal Kahler and Sasaki metrics},  arXiv:2104.09709v2

\bibitem{ACC+} C. Araujo, A. Castravet, I. Cheltsov, Kento fujita, A.Kaloghiros, J. Martinez-Garcia, C. Shramov, H.Suess, N. Viswanathan, \textit{The Calabi problem for Fano threefolds}, available at  \url{https://www.maths.ed.ac.uk/cheltsov/research.html}

\bibitem{BB} R. J. Berman and B. Berndtsson, \textit{Real Monge-Ampère equations and Kähler-Ricci solitons on toric log Fano varieties.} 
Ann. Fac. Sci. Toulouse, Math. (6) 22, No. 4, Spec. Issue, 649-711 (2013). 


\bibitem{Cal82} E. Calabi. \textit{Extremal Kähler metrics.} In Seminar on Differential Geometry, 102:259‑90. Ann. of Math. Stud. Princeton Univ. Press, Princeton, N.J., 1982. 

\bibitem{EC} E. Calabi, \textit{Extremal K\"ahler metrics, II}, Differential Geometry and Complex Analysis, eds. I. Chavel and H. M. Farkas, Springer Verlag (1985), 95–114.


\bibitem{XC1} X. Chen, J. Cheng, \textit{On the constant scalar curvature K\"ahler metrics: apriori estimates},  J. Amer. Math. Soc. (2021), https://doi.org/10.1090/jams/967.

\bibitem{XC2} X. Chen, J. Cheng, \textit{On the constant scalar curvature K\"ahler metrics II – Existence results}, J. Amer. Math. Soc. (2021), https://doi.org/10.1090/jams/966.

\bibitem{CCIII} X. Chen, J. Chen, \textit{On the constant scalar curvature Kähler metrics (III) -- general automorphism group}, arxiv:1801.05907 

\bibitem{KSSV2} T. Delcroix, \textit{Uniform K-stability of polarized spherical varieties}, arxiv:2009.06463.


\bibitem{TD} T. Delzant, \textit{Hamiltoniens périodiques et image convexe de l’application moment}, Bull. Soc.
Math. France \textbf{116} (1988), 315–339.


\bibitem{SKD}  S. K. Donaldson, \textit{Scalar curvature and stability of toric varieties}, J. Differential Geom. \textbf{62}
(2002), 289–349.

\bibitem{Don08} S. K. Donaldson, \textit{Kähler geometry on toric manifolds, and some other manifolds with large symmetry.} Adv. Lect. Math. (ALM) 7, 29--75, 2008.


\bibitem{FM} A. Futaki, T. Mabuchi,  \textit{Bilinear forms and extremal K\"ahler vector fields associated with K\"ahler classes.}
Math. Ann. \textbf{301} (1) (1995), 199-210.

\bibitem{Gau} P. Gauduchon, \textit{Calabi’s extremal metrics: An elementary introduction}, Lecture Notes

\bibitem{VG} V. Guillemin, \textit{K\"ahler structures on toric varieties}, J. Differential Geom. \textbf{40} (1994), 285–
309.

\bibitem{VGSS} V. Guillemin and S. Sternberg, \textit{Convexity properties of the moment mapping}, Invent. Math. \textbf{67}
(1982), 491–513.

 \bibitem{He} W. He, \textit{On Calabi's extremal metrics and properness}, Trans. Amer. Math. Soc. \textbf{372} (2019), 5595--5619
 
 \bibitem{TH} T. Hisamoto, \textit{ Stability and coercivity for toric polarizations}. arXiv:1610.07998v3,
2020


\bibitem{HS} A. D. Hwang and M. A. Singer. \textit{A momentum construction for circle-invariant Kähler metrics}.  Trans. Am. Math. Soc. 354, No. 6, 2285--2325, 2002.

\bibitem{SJ} S. Jubert, \textit{A Yau-Tian-Donaldson correspondance on a class of toric fibration},  arXiv:2108.12297v3 


 \bibitem{AL} A. Lahdili,  \textit{K\"ahler metrics with weighted constant scalar curvature and weighted K-stability}, Proc. London Math. Soc. (3), \textbf{119} (2019), 1065–114.
 
 \bibitem{LS93} C. LeBrun, S. R. Simanca. \textit{On the Kähler classes of extremal metrics.} In Geometry and
global analysis (Sendai, 1993), pages 255–271. Tohoku Univ., Sendai, 1993.
 
 \bibitem{Leg} E. Legendre, \textit{Toric Kähler-Einstein metrics and convex compact polytopes} J. Geom. Anal. 26, No. 1, 399--427 (2016)

\bibitem{Leg19} E. Legendre, \textit{A note on extremal toric almost Kähler metrics.} Springer INdAM Ser. 31, 53--74 (2019)

  \bibitem{Li} C. Li, \textit{Geodesic rays and stability in the cscK problem}, accepted by Annales Scientifiques de l'ENS, arXiv:2001.01366.
  
  
\bibitem{LH} C. Li, J. Han, \textit{On the Yau-Tian-Donaldson conjecture for generalized Kähler-Ricci soliton equations}, accepted by Comm. Pure Appl. Math., arXiv:2006.00903. 

 \bibitem{LLS} A. Li, Z. Lian, L. Sheng, \textit{Some Estimates for a Generalized Abreu’s Equation}, Differential Geom. Appl., \textbf{48} (2016), 87–103.
 
 \bibitem{LLS17} A. Li, Z. Lian, L. Sheng, \textit{Interior regularity for a Generalized Abreu Equation}, Int. J. Math., \textbf{108} (2017), 775–790.

 \bibitem{LLS21} A. Li, Z. Lian, L. Sheng, \textit{Extremal metrics on toric manifolds and homogeneous toric bundles}, arxiv:2110.08491


\bibitem{Mat57} Y. Matsushima. Sur la structure du groupe d’homéomorphismes analytiques
d’une certaine variété kählérienne. Nagoya Math. J., 11:145–150, 1957.


\bibitem{NS} Y. Nitta, S. Saito. \textit{A uniform version of the Yau-Tian-Donaldson correspondence for extremal Kähler metrics on polarized toric manifolds}, arXiv:2110.10386v1.



\bibitem{Sze07} G. Székelyhidi; \textit{Extremal metrics and K-stability.} Bull. Lond. Math. Soc. 39, No. 1, 76--84

\bibitem{SzeBook} G. Székelyhidi, \textit{An introduction to extremal Kähler metrics.} Providence, RI: American Mathematical Society (AMS) (2014)

\bibitem{WZ04} X.-J. Wang and X. Zhu, \textit{Kähler–Ricci solitons on toric manifolds with positive first Chern class.} Adv. Math. 188, No. 1, 87-103 (2004). 
 
\bibitem{ZZ} B. Zhou and X. Zhu, \textit{Relative K-stability and modified K-energy on toric manifolds}, Adv.
Math. \textbf{219} (2008), 1327–1362.

\end{thebibliography}
\end{document}